\documentclass{article}

\usepackage{geometry}
\usepackage{amsmath,amsfonts, amsthm,amssymb,oldgerm,amssymb}
\usepackage{graphicx}
\usepackage{array}
\usepackage{epsfig}
\usepackage{relsize}

\usepackage{hyperref}

\newcounter{indicee}
\newcounter{indiceK}
\newcounter{indiceC}
\def\omega{{\mathcal O}}
\def\R{{\mathbb  R}}

\def\l{\lambda}
\def\H{\mathcal{H}}

\newcommand{\D}{\mathcal{D}}
\newcommand{\CC}{\mathbb{C}}

\newcommand{\NN}{\mathbb{N}}

\newcommand{\RR}{\mathbb{R}}

\newcommand{\FF}{\mathbb{F}}

\newcommand{\h}{\hat}

\newcommand{\be}{\begin{eqnarray}}
\newcommand{\ee}{\end{eqnarray}}
\newcommand{\ben}{\begin{eqnarray*}}
\newcommand{\bc}{\begin{cases}}
\newcommand{\ec}{\end{cases}}
\newcommand{\een}{\end{eqnarray*}}
\newcommand{\ba}{\begin{array}}
\newcommand{\ea}{\end{array}}
\newcommand{\beq}{\begin{equation}}
\newcommand{\eeq}{\end{equation}}
\newcommand{\beqn}{\begin{equation*}}
\newcommand{\eeqn}{\end{equation*}}

\newcommand{\izl}{\int_0^L}
\newcommand{\izt}{\int_0^T}
\newcommand{\lzt}{L^2(0,T)}
\newcommand{\lzl}{L^2(0,L)}
\newcommand{\va}{\varphi}

\newcommand{\noi}{\noindent}

\newtheorem{theo}{Theorem}[section]

\newtheorem{prop}[theo]{Proposition}

\newtheorem{lemm}[theo]{Lemma}

\newtheorem{open}{Open Problem}[section]

\newtheorem{rem}[theo]{Remark}

\title{Boundary Controllability of the Korteweg-de Vries Equation on a Bounded Domain}
\author{
Eduardo Cerpa
\\ {\small Departamento de Matem\'atica} \\ {\small Universidad T\'ecnica Federico Santa Mar\'ia}\\
 {\small Casilla 110-V, Valpara\'iso, Chile} \\
 {\small email: eduardo.cerpa@usm.cl }\\ \qquad
\\Ivonne Rivas\\
{\small Instituto de Matem\'atica Pura e Aplicada (IMPA)} \  \\
 {\small Rio de Janeiro 22460-320, Brazil} \\
{\small
email: ivonriv@impa.br} \\
\quad\\ Bing-Yu Zhang \\
{\small Department of Mathematical Sciences} \\ {\small University
of Cincinnati}
\\ {\small Cincinnati,
Ohio 45221, USA} \\ {\small and} \\ {\small Yangtze Center of Mathematics}\\ {\small Sichuan University} \\  {\small Chengdu, China} \\
{\small email: zhangb@ucmail.uc.edu}}
\date{September 16, 2012}

\begin{document}

\maketitle
\newpage

\begin{abstract}
This paper is devoted to study boundary  controllability of the
Korteweg-de Vries equation posed on a finite interval, in which,
because of the third-order character of the equation, three boundary
conditions are required to secure the well-posedness of the system.
We consider the cases where one, two, or all three of those boundary
data are employed as boundary control inputs.  The system is first
linearized around the origin and the corresponding linear system is
shown to be exactly boundary controllable if using two or three
boundary control inputs. In the case where only one control input is
allowed to be used, the linearized system is known to be only
\emph{null} controllable if the single control input acts on the
left end of the spatial domain. By contrast, if the single control
input acts on the right end of the spatial domain, the linearized
system is exactly controllable if and only if the length of the
spatial domain does not belong to a set of \emph{critical values}.
Moreover, the nonlinear system is shown to be locally exactly
boundary controllable via \emph{contraction mapping principle} if
the associated linearized system is exactly controllable.
 \end{abstract}
\vspace{0.2 cm}


\vspace{0.2 cm} {\bf AMS subject classifications.} 93B05, 35Q53,
35Q53

\vspace{.2cm} {\bf Acknowledgments:} Eduardo Cerpa was partially
supported by Fondecyt Grant 11090161 and Programa Basal CMM, U. de
Chile, Ivonne Rivas was partially supported by CNPq and the Taft
Research Center at the University of Cincinnati,  and Bingyu Zhang
was partially supported by a grant from the Simons Foundation
(\#201615 to Bingyu Zhang) and NSF of China (\# 11231007).
\section{Introduction}\label{intro}

\quad
 In this paper we study a class of  distributed parameter  control systems described by
 the Korteweg-de Vries (KdV) equation posed
on a finite domain with nonhomogeneous boundary conditions:
\be\label{kdv}
    \begin{cases}
y_t + y_x + y_{xxx} + y y_x =0, \quad (x, t) \in (0,L)\times (0,T),\\
y(0,t)=h_1(t), \, y_x(L,t)=h_2(t), \, y_{xx}(L, t)=h_3(t), \quad t\in (0,T). \\
\end{cases} \ee \noi
This system  can be considered as a  model for propagation of
surface water waves in the situation where a wave-maker is putting
energy in a finite-length channel from the left $(x=0)$ while the
right end $(x=L)$ of the channel is free (corresponding to the case
 $h_2=h_3=0$) (see \cite{c-g-1}). Since the work of  Colin
and Ghidaglia in the late 1990's \cite{c-g-1,c-g-2, ColGhi01}, the
system (\ref{kdv}) has been mainly studied for its well-posedness in
the classical Sobolev space $H^s (0,L)$ \cite{kz,ruz}. So far, the
system is   known to be locally well-posed in the space
 $H^s (0,L)$ for any $s>-\frac34$ as stated in the following theorem.

\medskip
\noindent {\bf Theorem A }\cite{10RiKZ}: \emph{Let
$s>-\frac34$, $T>0$  and $r
>0$ be given with
\[ s\ne \frac{2j-1}{2}, \quad j=1,2,3,\cdots . \]
There exists  $T^*
>0$ such that for given $s-$compatible
\footnote{The reader is referred to \cite{10RiKZ} for the precise
definition of $s-$compatibility for the IBVP (\ref{kdv}).} data
$$y_0 \in H^s (0,L), \quad h_1
 \in H^{\frac{s+1}{3}} (0,T), \quad  h_2\in H^{\frac{s}{3}}(0,T),
 \quad h_3\in H^{\frac{s-1}{3}}(0,T) $$  satisfying
 \[ \|y_0 \|_{H^s(0,L)} +\| h_1\|_{H^{\frac{s+1}{3}}(0,T)} +\|h_2\|_{H^{\frac{s}{3}}(0,T)}
 +\|h_3\|_{H^{\frac{s-1}{3}}(0,T)} \leq r,\]
 then  (\ref{kdv}) admits a unique solution
 \[ y\in C([0,T^*]; H^s(0,L))\cap L^2 (0,T^*; H^{s+1} (0,L)) \]
 satisfying the initial condition
 \[ y|_{t=0}=y_0 .\]
 Moreover, the solution $y$ depends Lipschitz continuously on $y_0 $ and
 $h_j, j=1,2,3$ in the corresponding spaces.
}

\medskip
In this paper we are   interested in studying the IBVP (\ref{kdv})
from a control point of view:

\smallskip

 \emph{How solutions of
the system (\ref{kdv}) can be influenced by choosing appropriate
control inputs $h_j, \ j=1,2,3$?}

\smallskip

 In particular, we are
concerned with the following  exact boundary control problem.

\medskip
{\bf Exact Control Problem:}  \emph{Given $T>0$ and $y_0 , \ y_T
\in L^2 (0,L)$, can one find appropriate control inputs $h_j, \
j=1,2,3$ such that the corresponding solution $y$ of (\ref{kdv})
satisfies
\[ y|_{t=0}= y_0, \qquad y|_{t=T} =y_T?\]}

\medskip

Boundary control problems for  the KdV equation on a finite domain
have been extensively studied in the past (see
\cite{zh,97rosier,Za,04rosier,04coron-crepeau,07cerpa,08glass-guerrero,09cerpa-crepeau,
Za-survey} and the references therein). Most of those works have
been focused on the following system
 \be\label{kdv2}\left\{\ba{l}
u_t + u_x + u_{xxx} + u u_x =0,  \ (x,t)\in (0,L)\times (0,T), \\
u(0,t)=g_1(t), \, u(L,t)=g_2(t), \, u_{x}(L,t)=g_3(t),\quad t\in (0,T)
 \ea \right.\ee which possess a different
set of boundary conditions than those of the system (\ref{kdv}).
Controllability of this system was first studied by Rosier
\cite{97rosier} in 1997 using only one control input $g_3$:
 \be\label{kdv2-1}\left\{\ba{l}
u_t + u_x + u_{xxx} + u u_x =0,  \ (x,t)\in (0,L)\times (0,T), \\
u(0,t)=0, \, u(L,t)=0, \, u_{x}(L,t)=g_3(t), \quad t\in (0,T).
 \ea \right.\ee
 It was discovered rather surprisingly that whether the associated linear
 system
 \be\label{kdv2-1-1}\left\{\ba{l}
u_t + u_x + u_{xxx}  =0,  \ (x,t)\in (0,L)\times (0,T), \\
u(0,t)=0, \, u(L,t)=0, \, u_{x}(L,t)=g_3(t), \quad t\in (0,T),
 \ea \right.\ee
 is exactly controllable depends on the length $L$ of the spatial domain $(0,L)$.  More precisely,
Rosier \cite{97rosier} showed that \emph{the linear system is
exactly controllable in the space $L^2 (0,L)$ if and only if
\begin{equation} \label{NN-1} L \notin {\cal S}:= \left\{ 2 \pi
\sqrt{\frac{k^2+kl+l^2}{3}} ;\, k,l \in \NN^* \right\}.
\end{equation}}
 With the linear result in hand and using the
  \emph{contraction mapping principle}, Rosier \cite{97rosier} showed further that the nonlinear
  system (\ref{kdv2-1}) is locally exactly controllable
 in the space $L^2 (0,L)$ so long as $L\notin {\cal S}$.

\medskip
\noindent {\bf Theorem B} (Rosier \cite{97rosier}): \emph{Let
$T>0$ be given and assume $L\notin {\cal S}$. There exists  $r>0$
such that for any $u_0, \ u_T \in L^2 (0,L)$ with \[ \|u_0 \|_{L^2
(0,L)}+\|u_T \|_{L^2 (0,L)} \leq r,\] there exists $g_3\in L^2
(0,T) $  such that the system (\ref{kdv2-1})  admits a unique
solution
\[ u\in C([0,T]; L^2 (0,L))\cap L^2 (0,T; H^1 (0,L))\]
satisfying
\[ u|_{t=0}= u_0, \qquad u|_{t=T} =u_T.\]}

\medskip
The system (\ref{kdv2}) was later studied by Glass and Guerrero
\cite{09glass-guerrero} for its  boundary controllability using only
$g_2$ as a control input.
 \be\label{kdv2-2}\left\{\ba{l}
u_t + u_x + u_{xxx} + u u_x =0,  \ (x,t)\in (0,L)\times (0,T), \\
u(0,t)=0, \, u(L,t)=g_2(t), \, u_{x}(L,t)=0, \quad t\in (0,T).
 \ea \right.\ee
 They showed that   the corresponding
 linear system
  \be\left\{\ba{l}
u_t + u_x + u_{xxx}  =0,  \ (x,t)\in (0,L)\times (0,T), \\
u(0,t)=0, \, u(L,t)=g_2(t), \, u_{x}(L,t)=0, \quad t\in (0,T),
 \ea \right.\ee
  is exactly controllable in the
space $L^2 (0,L)$ if and only if $L\notin {\cal N}$ where
\begin{eqnarray} {\cal N} & = & \Big\{ L\in \RR^+:
L^2=-(a^2+ab+b^2) \ \text{with} \ a,b \in \CC  \nonumber
\\ & & \qquad    \qquad  \text{satisfying}
  \quad
 ae^a  =   be^b
= -(a+b) e^{-(a+b)} \Big\}. \label{n-1}\end{eqnarray} Then the
nonlinear system (\ref{kdv2-2}) was shown to be  locally exactly
controllable in the space $L^2 (0,L)$ if $L \notin {\cal N}$.

\medskip
\noindent {\bf Theorem C} (Glass and Guerrero
\cite{09glass-guerrero})   \emph{Let $T>0$ and $L\notin {\cal N}$
be given. There exists  $r>0$ such that for any $u_0, \ u_T \in
L^2 (0,L)$ with \[ \|u_0 \|_{L^2 (0,L)}+\|u_T \|_{L^2 (0,L)} \leq
r,\] one can find    $g_2\in H^{\frac16-\epsilon} (0,T) $ for any $\epsilon>0$ such
that the system (\ref{kdv2-2})   admits a solution
\[ u\in C([0,T]; L^2 (0,L))\cap L^2 (0,T; H^1 (0,L))\]
satisfying
\[ u|_{t=0}= u_0, \qquad u|_{t=T} =u_T.\]}

\medskip
While the critical length phenomenon occurs when a single control
input (either $g_2$ or $g_3$) is used,  it will not happen, however,
if more than one control inputs are allowed to be used. It was already
pointed out  in \cite{97rosier} that the linear system associated to
(\ref{kdv2}) is exactly controllable  for any $L>0$ if both $g_2 $
and $g_3$ are allowed to be used as control inputs. Moreover,
 the nonlinear system
\be\label{kdv2--2}\left\{\ba{l}
u_t + u_x + u_{xxx} + u u_x =0,  \ (x,t)\in (0,L)\times (0,T), \\
u(0,t)=0, \, u(L,t)=g_2(t), \, u_{x}(L,t)=g_3(t), \quad t\in (0,T),
 \ea \right.\ee was shown in \cite{97rosier} to be locally exactly
 controllable in $L^2 (0,L)$.

\medskip
\noindent {\bf Theorem D} (Rosier \cite{97rosier}): \emph{Let $T>0 $ and $L>0$ be given and
$k>0$ be an integer. There exists  $\delta >0$ such that for any $u_0, \
u_T \in
L^2 (0,L)$ with \[ \|u_0 \|_{L^2 (0,L)}+\|u_T \|_{L^2
(0,L)} \leq \delta ,\] there exist  $g_2\in H^k_0 (0,T)$ and
$g_3\in L^2 (0,T) $ such that the system (\ref{kdv2--2}) admits a
solution
\[ u\in C([0,T]; L^2 (0,L))\cap L^2 (0,T;H^1(0,L))\]
satisfying
\[ u|_{t=0}= u_0, \qquad u|_{t=T} =u_T.\]}

In \cite{09glass-guerrero} Glass and Guerrero considered the
system (\ref{kdv2}) using $g_1$ and $g_2$ as control inputs.
\be\label{kdv2-3}\left\{\ba{l}
u_t + u_x + u_{xxx} + u u_x =0,  \ (x,t)\in (0,L)\times (0,T), \\
u(0,t)=g_1, \, u(L,t)=g_2(t), \, u_{x}(L,t)=0, \quad t\in (0,T),
 \ea \right.\ee
 and showed that the system is also locally exactly controllable for any
 $L>0$.

 \medskip
 \noindent
 {\bf Theorem E} (Glass and Guerrero
\cite{08glass-guerrero}) \emph{Let $T>0$ and $L>0$ be given. There
exists $r>0$ such that for any $u_0, \ u_T \in L^2 (0,L)$ with \[
\|u_0 \|_{L^2 (0,L)}+\|u_T \|_{L^2 (0,L)} \leq r,\] there exist
$g_1$ and  $g_2\in L^2 (0,T) $ such that the system (\ref{kdv2-3})
admits a solution
\[ u\in C([0,T]; H^{-1} (0,L))\cap L^2 (0,T; L^2 (0,L))\]
satisfying
\[ u|_{t=0}= u_0, \qquad u|_{t=T} =u_T.\]}

\medskip
Another interesting  controllability result  regarding the system
(\ref{kdv2}) is that it is only \emph{null controllable} if the control acts from the left side of the spatial
domain $(0,L)$, which was proved by Rosier \cite{04rosier}, Glass
and Guerrero \cite{08glass-guerrero}.

\medskip\emph{
\noindent {\bf Theorem F}(\cite{04rosier, 08glass-guerrero})   Let $T>0$ and $L>0$ be given.  Let
$$v\in C([0,T];L^2 (0,L))\cap L^2 (0,T;H^1 (0,L))$$ satisfy
} \beq \begin{cases}
v_t+v_x+vv_x+v_{xxx}=0, \ v(x,0)=v_0(x), \qquad  (x,t)\in (0,L)\times (0,T), \\
 v(0,t)=v(L,t)=v_x(L,t)=0,\quad t\in(0,T). \\
 \end{cases}
  \eeq
\emph{Then, there exists  $\delta >0$ such that for any $u_0 \in L^2
(0,L)$
    with}
    \[ \|u_0 -v_0\|_{L^2 (0,L)} \leq \delta , \]
   \emph{ there exists $g_1\in H^{\frac12 -\epsilon}(0,T)$ for any $\epsilon
    >0$ such that the system (\ref{kdv2}) admits a solution}
    $$u\in C([0,T]; L^2 (0,L))\cap L^2(0,T;H^1(0,L))$$ satisfying
      \[ u|_{t=0}=u_0, \qquad u|_{t=T} =v|_{t=T} .\]

While boundary controllability of the system (\ref{kdv2}) has been
well studied, there is very few results for the system (\ref{kdv}).
To our knowledge, the only known result is due to Guilleron
\cite{guilleron}. He has considered the   system (\ref{kdv}) with
$h_2=h_3=0$ and shown the corresponding linear system
\[
\begin{cases}
y_t+y_x+y_{xxx}=0, \qquad (x,t)\in (0,L)\times (0,T),\\
y(0,t)= h_1 (t), \qquad y_x(L,t)=y_{xx}(L,t)=0 \end{cases} \] is
null controllable  by applying a Carleman estimates approach to
obtain the needed observability inequality.

The purpose of this paper is to fill the gap and to determine if the
system (\ref{kdv}) possesses controllability results similar to
those established  for system (\ref{kdv2}). Naturally one would like
to try the same approaches that have worked effectively for system
(\ref{kdv2}). However, one will encounter some difficulties that
demand special attention and some new tools will be needed. In
particular, when we  use only $h_2 $ as a control input, the linear
system associated to (\ref{kdv}) is \be\label{kdv-11}
    \begin{cases}
y_t + y_x + y_{xxx} =0, \quad y(x,0)=y_0 (x), \quad (x, t) \in (0,L)\times (0,T),\\
y(0,t)=0, \, y_x(L,t)=h_2(t), \, y_{xx}(L, t)=0, \quad t\in(0,T) \\
\end{cases} \ee \noi
 and its adjoint system is given by
  \be\label{aj-1}
    \begin{cases}
\va _t + \va_x + \va_{xxx} =0, \quad \va(x,T)=\va_T (x), \quad (x, t) \in (0,L)\times (0,T),\\
\va (L,t)+\va _{xx} (L,t)=0, \, \va _x(0,t)=0, \, \va (0, t)=0, \quad t\in(0,T). \\
\end{cases} \ee \noi
It is well-known that the exact controllability of system
(\ref{kdv-11}) is equivalent to the following observability inequality for the adjoint
system \eqref{aj-1}  \be\label{e-11}
 \|\va
_T\|_{L^2 (0,L)} \leq C \|\va _x (L, \cdot)\|_{L^2 (0,T)}. \ee
However, the usual multiplier method  and compactness arguments as
those used in dealing with the control of system (\ref{kdv2-1-1})
only lead  to \be\label{e-2} \|\va _T\|_{L^2 (0,L)} \leq C_1 \|\va
_x (L, \cdot)\|_{L^2 (0,T)} + C_2 \| \va (L,\cdot )\|_{L^2 (0,T)}.
\ee
 How to remove  the extra term in (\ref{e-2}) presents  a
challenge and demands a  new tool. This new   tool turns out to be
the hidden  regularity (or the sharp Kato smoothing property
\cite{Kato83}) for solutions of the KdV equation. Specially, as we
will demonstrate later in this paper, for solutions of the system
(\ref{aj-1}), the following inequality holds  \be \label{e-3} \sup
_{0<x<L} \| \va (x, \cdot )\|_{H^{\frac13} (0,T)} \leq C \|\va
_T\|_{L^2 (0, L)} ,\ee which will play  a crucial role in validating
the observability estimate (\ref{e-11}).

 In this paper, we will first consider  case that only   $h_2$ is employed  as a control input and show that the system
\be\label{kdv-1.1}
    \begin{cases}
y_t + y_x + y_{xxx} + y y_x =0,  (x, t) \in (0,L)\times (0,T),\\
y(0,t)=0, \, y_x(L,t)=h_2(t), \, y_{xx}(L, t)=0,   \quad t\in(0,T)
\end{cases} \ee \noi
 is \emph{locally exactly controllable} as
long as  $L\notin \FF$  where \beq\label{FF} \FF = \left\{ L\in
\RR^+: L^2=-(a^2+ab+b^2) \ \text{with} \ a,b \in \CC \
\text{satisfying}
  \quad
\frac{e^a}{a^2}= \frac{e^b}{b^2}
=\frac{e^{-(a+b)}}{(a+b)^2}\right\}.\eeq

\begin{theo}\label{theo-h-2} Let $T>0$ and $L\notin \FF$ be given.  There exists  $\delta >0$ such
that for any $y_0, \ y_T \in L^2 (0,L)$ with \[ \|y_0\|_{L^2
(0,L)}+\|y_T\|_{L^2 (0,L)} \leq \delta ,\] one can find $h_2\in
L^2 (0,T)$ such that the system (\ref{kdv-1.1}) admits a unique
solution
\[ y\in C([0,T]; L^2 (0,L))\cap L^2 (0,T; H^1 (0,L))\]
satisfying
\[ y|_{t=0}= y_0, \qquad y|_{t=T} =y_T.\]
\end{theo}

\begin{rem} It can be proven that the set $\FF$ is nonempty and countable.
 The proof follows \cite{09glass-guerrero} where the authors proved that the set $\cal{N}$ (see \eqref{n-1}) is nonempty and countable.
\end{rem}

Instead of employing control input $h_2$, one can just
use the control input $h_3$: \be\label{kdv-1.2}
    \begin{cases}
y_t + y_x + y_{xxx} + y y_x =0,  (x, t) \in (0,L)\times (0,T),\\
y(0,t)=0, \, y_x(L,t)=0, \, y_{xx}(L, t)=h_3(t),  \quad t\in(0,T).
\end{cases} \ee \noi
The corresponding system  is also locally exactly controllable if
the length $L$ of the interval $(0,L)$ does not belong to the set
${\cal N}$ as defined in (\ref{n-1}).
\begin{theo}\label{theo-h-3} Let $T>0$ and $L\notin  {\cal N}$ be given.  There exists  $\delta >0$ such
that for any $y_0, \ y_T \in L^2 (0,L)$ with \[ \|y_0\|_{L^2
(0,L)}+\|y_T\|_{L^2 (0,L)} \leq \delta ,\] there exists $h_3\in
H^{-\frac13} (0,T)$ such that the system (\ref{kdv-1.2}) admits a
unique solution
\[ y\in C([0,T]; L^2 (0,L))\cap L^2 (0,T; H^1 (0,L))\]
satisfying
\[ y|_{t=0}= y_0, \qquad y|_{t=T} =y_T.\]
\end{theo}

Similar to the system (\ref{kdv2}), the  critical length
phenomenon will not occur if more than one control inputs are
employed. For the system where $h_1$ and $h_2$ are used as control
inputs,
\be\label{kdv-1.3}
    \begin{cases}
y_t + y_x + y_{xxx} + y y_x =0,  (x, t) \in (0,L)\times (0,T),\\
y(0,t)=h_1(t), \, y_x(L,t)=h_2(t), \, y_{xx}(L, t)=0, \quad t\in(0,T)  \\
\end{cases} \ee \noi
we have the following local exact controllability result.

\begin{theo}\label{theo-h-1-2} Let $T>0$  and $L>0$ be given. There exists $\delta >0$ such that
for any $y_0, \ y_T \in L^2 (0,L)$ with \[ \|y_0\|_{L^2
(0,L)}+\|y_T\|_{L^2 (0,L)} \leq \delta ,\] one can find  $h_1\in
H^{\frac13}(0,T) $ and $h_2\in L^2 (0,T)$ such that the system
(\ref{kdv-1.3})  admits a unique solution
\[ y\in C([0,T]; L^2 (0,L))\cap L^2 (0,T; H^1 (0,L))\]
satisfying
\[ y|_{t=0}= y_0, \qquad y|_{t=T} =y_T.\]
\end{theo}
For the system using both $h_2$ and $h_3$ as control inputs,
\be\label{kdv-1.4}
    \begin{cases}
y_t + y_x + y_{xxx} + y y_x =0,  (x, t) \in (0,L)\times (0,T),\\
y(0,t)=0, \, y_x(L,t)=h_2(t), \, y_{xx}(L, t)=h_3(t), \quad t\in(0,T) , \\
\end{cases} \ee \noi
we have the following local exact controllability result.

\begin{theo}\label{theo-h-2-3} Let $T>0$   and $L>0$ be given. There exists  $\delta >0$ such that
for any $y_0, \ y_T \in L^2 (0,L)$ with \[ \|y_0\|_{L^2
(0,L)}+\|y_T\|_{L^2 (0,L)} \leq \delta ,\] one can find $h_3\in
H^{-\frac13}(0,T) $ and $h_2\in L^2 (0,T)$ such that the system
(\ref{kdv-1.4}) admits a unique solution
\[ y\in C([0,T]; L^2 (0,L))\cap L^2 (0,T; H^1 (0,L))\]
satisfying
\[ y|_{t=0}= y_0, \qquad y|_{t=T} =y_T.\]
\end{theo}
For the system using $h_1$ and $h_3$ as control inputs,
\be\label{kdv-1.5}
    \begin{cases}
y_t + y_x + y_{xxx} + y y_x =0,  (x, t) \in (0,L)\times (0,T),\\
y(0,t)=h_1(t), \, y_x(L,t)=0, \, y_{xx}(L, t)=h_3(t), \quad t\in(0,T)\\
\end{cases} \ee \noi
we have
\begin{theo}\label{theo-h-1-3} Let $T>0$  and $L>0$ be given. There exists  $\delta >0$ such that
for any $y_0, \ y_T \in L^2 (0,L)$ with \[ \|y_0\|_{L^2
(0,L)}+\|y_T\|_{L^2 (0,L)} \leq \delta ,\] one can find $h_1\in
H^{\frac13}(0,T) $ and $h_3\in H^{-\frac13}(0,T)$ such that the
system (\ref{kdv-1.5})   admits a unique solution
\[ y\in C([0,T]; L^2 (0,L))\cap L^2 (0,T; H^1 (0,L))\]
satisfying
\[ y|_{t=0}= y_0, \qquad y|_{t=T} =y_T.\]
\end{theo}

If all three boundary control inputs are allowed to be used, then we
can show that system (\ref{kdv}) is locally exactly
controllable around any smooth solution of the KdV equation.
\begin{theo}\label{theo-h-1-2-3} Let $T>0$ and $L>0$ be given. Assume  that
    $u \in C^{\infty}(\R, H^{\infty} (\R))$ satisfies
    $$u_t+u_x+uu_x+u_{xxx}=0, \quad (x,t)\in  \R\times \R.$$
    Then there exists  $\delta >0$  such that for any $y_0, y_T \in L^2(0,L),$ satisfying
    $$\| y_0 - u(\cdot,0)\|_{L^2 (0,L) }+ \| y_T - u(\cdot,T)\|_{L^2(0,L)}
    \le \delta $$
    one can find control inputs $$h_1\in H^{\frac13} (0,T), \quad h_2\in L^2 (0,T),\quad h_3\in
     H^{-\frac13}(0,T)$$
    such that \eqref{kdv} admits a unique solution
    \[ y\in C([0,T]; L^2 (0,L))\cap L^2 (0,T; H^1 (0,L))\]
satisfying
\[ y|_{t=0}= y_0, \qquad y|_{t=T} =y_T.\]
\end{theo}

\medskip
Finally, with the  help of some hidden regularity properties for solutions of
the KdV equation we can improve some controllability results
for the system (\ref{kdv2}). Concerning
 system \be\label{kdv2-2-1}\left\{\ba{l}
u_t + u_x + u_{xxx} + u u_x =0,  \ (x,t)\in (0,L)\times (0,T), \\
u(0,t)=0, \, u(L,t)=g_2(t), \, u_{x}(L,t)=0, \quad t\in(0,T)
 \ea \right.\ee we can show that the control input $g_2$, used in Theorem C,
   belongs in fact to the space $H^{\frac13}(0,T)$.

\begin{theo}\label{ex-1} Let $T>0$ and $L\notin {\cal N}$ be given. There exists
$r>0$ such that for any $u_0, \ u_T \in L^2 (0,L)$ with \[ \|u_0
\|_{L^2 (0,L)}+\|u_T \|_{L^2 (0,L)} \leq r,\] there exists
$g_2\in H^{\frac13} (0,T) $ such that the system (\ref{kdv2-2-1})
admits a solution
\[ u\in C([0,T]; L^2 (0,L))\cap L^2 (0,T; H^1 (0,L))\]
satisfying
\[ u|_{t=0}= u_0, \qquad u|_{t=T} =u_T.\]
\end{theo}

\medskip
Regarding system \be\label{kdv2--2-1}\left\{\ba{l}
u_t + u_x + u_{xxx} + u u_x =0,  \ (x,t)\in (0,L)\times (0,T), \\
u(0,t)=0, \, u(L,t)=g_2(t), \, u_{x}(L,t)=g_3(t),\quad t\in(0,T),
 \ea \right.\ee
the Theorem D can be improved as follows.

\medskip
\begin{theo}\label{ex-2}
Let $T>0 $ and $L>0$ be given. There exists $\delta >0$ such that for any
$u_0, \ u_T \in L^2 (0,L)$ with \[ \|u_0 \|_{L^2 (0,L)}+\|u_T
\|_{L^2 (0,L)} \leq \delta ,\] there exist  $g_2 \in H^{\frac13}
(0,T)$ and $ g_3\in L^2 (0,T) $ such that system (\ref{kdv2--2-1})
admits a solution
\[ u\in C([0,T]; L^2 (0,L))\cap L^2 (0,T;H^1(0,L))\]
satisfying
\[ u|_{t=0}= u_0, \qquad u|_{t=T} =u_T.\]
\end{theo}

\medskip
Moreover, we also consider  system (\ref{kdv2}) using only two
control inputs $g_1$ and $g_2$ \be\label{kdv2--3-1}\left\{\ba{l}
u_t + u_x + u_{xxx} + u u_x =0,  \ (x,t)\in (0,L)\times (0,T), \\
u(0,t)=g_1 (t), \, u(L,t)=g_2(t), \, u_{x}(L,t)=0, \quad t\in(0,T)
 \ea \right.\ee
which has not been studied in the literature before.  We can show
that  the critical length phenomenon will not occur for this
system neither.

\medskip
\begin{theo}\label{ex-3}
Let $T>0 $ and $L>0$ be given. There exists  $\delta >0$ such that for any
$u_0, \ u_T \in L^2 (0,L)$ with \[ \|u_0 \|_{L^2 (0,L)}+\|u_T
\|_{L^2 (0,L)} \leq \delta ,\] one can find $g_1 \in H^{\frac13}
(0,T)$ and  $g_2\in H^{\frac13}(0,T) $ such that the system
(\ref{kdv2--3-1}) admits a solution
\[ u\in C([0,T]; L^2 (0,L))\cap L^2 (0,T;H^1(0,L))\]
satisfying
\[ u|_{t=0}= u_0, \qquad u|_{t=T} =u_T.\]
\end{theo}

The paper is organized as follows.
 In Section \ref{sec-est}, we present  various linear estimates including
hidden regularities for solutions of the linear systems associated
to (\ref{kdv}) and (\ref{kdv2}) which will play important roles in
establishing our exact controllability results in this paper. The
associated linear systems are shown to be exactly controllable in
Section \ref{sec-lin} while the nonlinear systems are shown to be locally
exactly controllable using the standard \emph{contraction mapping
principle} in Section \ref{sec-non}.   Finally, in Section \ref{conclusion} we
provide some conclusion remarks together with some  open problems
for further studies. The paper is ended with an appendix where
the proofs of some technical lemmas used in the paper are
furnished.


\section{Linear estimates}\label{sec-est}
\setcounter{equation}{0}
\subsection{The forward linear system}

\quad

In this subsection, we consider  the following linear problem
associated to the nonlinear system (\ref{kdv})
 \be\label{Gkdv}
\begin{cases}
y_t + y_x + y_{xxx} =f, \quad  \ y(x,0)= y_0 (x), \quad  x \in (0,L),\  t>0,\\
y(0,t)=h_1(t), \, y_x(L,t)=h_2(t), \,
y_{xx}(L,t)=h_3(t),\quad t>0.\end{cases} \ee In the case $h_1=h_2=h_3=0$
and $f=0$, the solution $y$ can be written as
\[ y= W_0 (t) y_0. \]
Here  $W_0(t)$ is  the $C_0$-semigroup in the space $L^2(0,L)$ (see \cite{Pazy})
generated by the linear operator
$$A\psi=-\psi '''- \psi'$$
whose domain is
$$\D(A)=\{ \psi \in H^3(0,L): \psi (0)=\psi '(L)=\psi ''(L)=0  \}.$$
The solution $y$ of \eqref{Gkdv}, when $h_1=h_2=h_3=0$ and $y_0=0$, is given by  \[ y=\int ^t_0 W_0 (t-\tau ) f(\tau ) d\tau \]
while  in the case $y_0=0$ and $f=0$,  it has the form
\[ y= W_{bdr} (t) \vec{h}\]
where $\vec{h} = (h_1, h_2 , h_3)$ and $W_{bdr}(t)$ is the
associated  boundary integral operator defined in
\cite{kz,10RiKZ}.

\medskip
 As it has been demonstrated in \cite{kz,10RiKZ} the linear system (\ref{Gkdv}) is  well-posed in the space $H^s (0,L)$
 for any $0\leq s\leq 3$  with
 \[ y_0\in H^s (0,L), \ f\in W^{\frac{s}{3},1}(0,T;L^2(0,L))\] and
 \[ \vec{h}=(h_1, h_2 , h_3)\in {\cal H}^s_{loc} (\R^+) :=
  H^{\frac{s+1}{3}}_{loc} (\R^+)\times H^{\frac{s}{3}}_{loc}(\R^+)\times H^{\frac{s-1}{3}}_{loc} (\R^+).\]
 In particular, in the case $s=0$,  the result can be stated as
 follows.

 \begin{prop}\label{allsys}
 Let $T>0$ be given,  for any $y_0\in L^2(0,L)$, $f\in
 L^1(0,T;L^2(0,L))$ and
 $$(h_1,h_2,h_3)\in \H _T:=H^{\frac13}(0,T)\times L^2(0,T)\times H^{-\frac13}(0,T),$$
 the IBVP \eqref{Gkdv}  admits a unique solution
 $$y\in X_T:=C([0,T];L^2(0,L))\cap L^2(0,T;H^1(0,L)).$$
Moreover,  there exists $C>0$ such that
 $$\|y\|_{X_T} \le  C\left(\|f\|_{L^1(0,T;L^2(0,L))} + \|y_0\|_{L^2(0,L)} + \|(h_1,h_2,h_3)\|_{\H _T} \right).$$
 \end{prop}
In addition, the solution $y$ of (\ref{Gkdv}) possesses the
 following hidden (or sharp trace) regularities.
 \begin{prop}\label{traces} Let $T>0$ be given.  For any $y_0\in L^2(0,L)$, $f\in
 L^1(0,T;L^2(0,L))$ and
 $(h_1,h_2,h_3)\in \H _T,$
 the solution $y$ of the system (\ref{Gkdv}) satisfies
 \be \label{h-1}\sup _{0<x<L} \| \partial _x^{j} y (x,
 \cdot)\|_{H^{\frac{1-j}{3}} (0,T)} \leq C_j \left(\|f\|_{L^1(0,T;L^2(0,L))} + \|y_0\|_{L^2(0,L)} + \|(h_1,h_2,h_3)\|_{\H _T}
 \right)\ee
 for $j=0,1,2$.
 \end{prop}
 \noindent
 Next proposition states similar hidden (or sharp trace) regularity results for the linear system
 \be\label{2.2}
\begin{cases}
u_t + u_x + u_{xxx} =f, \quad  \ u(x,0)= u_0 (x), \quad  x \in (0,L),\  t>0,\\
u(0,t)=g_1(t), \, u(L,t)=g_2(t), \, u_{x}(L,t)=g_3(t),\quad t>0,\end{cases}
\ee  associated to  (\ref{kdv2}).
\begin{prop}\label{regular}
Let $T>0$ be given,  for any $u_0\in L^2(0,L)$, $f\in
 L^1(0,T;L^2(0,L))$ and
 $$(g_1,g_2,g_3)\in {\cal G} _T:=H^{\frac13}(0,T)\times H^{\frac13}(0, T)\times L^2(0,T),$$
 the IBVP \eqref{2.2}  admits a unique solution
 $u\in X_T.$ Moreover,  there exists $C>0$ such that
 $$\|u\|_{X_T} \le  C\left(\|f\|_{L^1(0,T;L^2(0,L))} + \|y_0\|_{L^2(0,L)} + \|(g_1,g_2,g_3)\|_{{\cal G} _T} \right).$$
 In addition, the solution $u$ possesses the following sharp trace
 estimates
 \be \label{h-2} \sup _{0<x<L} \| \partial _x^{j} u (x,
 \cdot)\|_{H^{\frac{1-j}{3}} (0,T)} \leq C_j \left(\|f\|_{L^1(0,T;L^2(0,L))} + \|u_0\|_{L^2(0,L)} + \|(g_1,g_2,g_3)\|_{{\cal G}_T}
 \right)\ee
 for $j=0,1,2$.

 \end{prop}
 The proofs of Proposition \ref{traces} and Proposition
 \ref{regular} can be found in \cite{zhang-1} (cf. also \cite{03bona-sun-zhang,bsz-finite,10RiKZ}).

 \begin{rem}
 Systems (\ref{Gkdv}) and (\ref{2.2}) are equivalent
 in the following sense: for given $\{ y_0, f, h_1,h_2,h_3\} $ one
 can find $\{ u_0, f, g_1,g_2,g_3\} $ such that the corresponding
 solution $y$ of (\ref{Gkdv}) is exactly the same as the corresponding $u$ for  system
 (\ref{2.2}) and vice versa. Indeed, for given $y_0\in L^2 (0,L)$, $f\in L^1 (0,T; L^2 (0,L))$
 and $\vec{h}\in {\cal H}_T$, system (\ref{Gkdv}) admits a unique
 solution $y\in X_T$. Let
 $ u_0=y_0$, and set
 \[ g_1(t)=h_1 (t), \quad g_2 (t) = y(L,t), \quad g_3 (t)=
 h_2(t).\]
 Then, according to (\ref{h-1}), we have $\vec{g} \in {\cal G}_T$.
 Because of the uniqueness of the IBVP (\ref{2.2}),
 with such selected $(u_0, f, g_1,g_2, g_3)$, the corresponding solution $u\in X_T$ of (\ref{2.2}) must be equal
 to $y$ since $y$ also solves (\ref{2.2}) with the given auxiliary data $(u_0, f, g_1,g_2,
 g_3)$. On the other hand, for any given $u_0\in L^2 (0,L)$, $f\in
 L^1 (0,T; L^2 (0,L))$ and $\vec{g}\in {\cal G}_T$, let $u\in X_T$
 be the corresponding solution of the system (\ref{2.2}).  By
 (\ref{h-2}), we have $u_{xx} (L, \cdot )\in H^{-\frac13} (0,T)$.
 Thus, if set $y_0=u_0$ and
 \[  h_1(t)=g_1(t), \quad h_2(t)=g_3 (t), \quad h_3 (t)=u_{xx}
 (L,t),\]
 then $\vec{h}\in {\cal H}_T$ and the corresponding solution  $y\in X_T$ of
 (\ref{Gkdv}) must be equal to $u$ which also solves (\ref{Gkdv})
 with the auxiliary data $(y_0, f, \vec{h})$.
 \end{rem}

 \subsection{The backward adjoint linear system}

 \quad
 In this subsection, we consider the  backward adjoint  system  of
 (\ref{Gkdv})
\be\label{10-1} \bc
\psi_t + \psi_x +  \psi_{xxx}  =0, \quad  \psi(x,T)= \psi_T(x), \quad (x, t)\in (0,L)\times (0,T),\\
\psi (0,t)=0, \quad  \psi_x(0,t)=0, \quad   \psi
(L,t)+\psi_{xx}(L,t)=0,\quad t\in (0,T) , \ec\ee  which (by
transformation $x'=L-x,\ t'=T-t$) is equivalent to the following
forward system \be\label{2.3} \bc
\varphi_t + \varphi_x +  \varphi_{xxx}  =0, \quad \varphi(x,0)= \varphi_0(x),\quad (x, t)\in (0,L)\times (0,T),\\
\varphi(L,t)=0, \quad  \varphi_x(L,t)=0, \quad
\varphi(0,t)+\varphi_{xx}(0,t)=0, \quad t\in (0,T).
 \ec\ee  The solution  of (\ref{2.3}) can
be written as
\[ \va  (x,t) =S(t) \va_0 \]
where $S(t)$ is the $C_0$ semigroup in the space $L^2 (0,L)$
generated by the operator
\[ A_1 f= -f'-f'''\]
with the domain
\[ {\cal D} (A_1)=\{ f\in H^3 (0,L):\ f(0)+f''(0)=0, \
f(L)=f'(L)=0\} .\]

\begin{prop}\label{fistpart}
For any $\va_0\in L^2(0,L)$ the IBVP \eqref{2.3} admits a unique
solution $\va \in  X_T$. Moreover,  there exists $C>0$ such that
\[ \| \va\|_{X_T} \leq C \|\va_0\|_{L^2(0,L)}  \]
 and \[
\int^T_0 \left ( |\va (0,t)|^2 + |\va _x (0,t)|^2 \right ) dt \leq C
\| \va _0\|^2_{L^2 (0,L)} .\]

\end{prop}
\noindent {\bf Proof.} The proof is very similar to that of \cite{97rosier} and is therefore
omitted.
$\blacksquare$

\medskip
Thus, when $\va _0 \in L^2 (0,L)$, the corresponding solution $\va
$ has the trace $\va (L, \cdot )\in L^2 (0,T)$. The next theorem
reveals that  $\va $ has a stronger trace regularity: $\va (L,
\cdot )\in H^{\frac13} (0,T)$. It will play an important role to
establish exact controllability of the system (\ref{kdv}) as shown
in the next section.

\begin{theo}[Hidden regularities]\label{adprop1}
For any $\va_0 \in L^2(0,L)$, the solution $\varphi \in X_T$ of
IBVP \eqref{2.3} possesses the following sharp trace properties
\[ \sup
_{0<x<L} \| \partial _x^{j} \va (x,
 \cdot)\|_{H^{\frac{1-j}{3}} (0,T)} \leq C_j   \|\va_0\|_{L^2(0,L)}\]
 for $j=0,1,2$.

\end{theo}
\begin{rem} Equivalently the solutions of the system (\ref{2.2})
has the following sharp trace estimates:
\[ \sup
_{0<x<L} \| \partial _x^{j} \psi (x,
 \cdot)\|_{H^{\frac{1-j}{3}} (0,T)} \leq C_j   \|\psi_T\|_{L^2(0,L)}\]
 for $j=0,1,2$.
\end{rem}

To prove Theorem \ref{adprop1}, we first consider the following
linear system
\be \label{2.4} \bc
w_t + w_{xxx}  =f, \quad (x,t) \in (0,L)\times (0,T),\\
w_{xx}(0,t)=k_1(t), \quad  w(L,t)=k_2(t), \quad  w_x(L, t)=k_3(t),\quad t\in (0,T) ,\\
w(x,0)= w_0(x), \quad x\in (0,L).  \ec \ee

\begin{prop}\label{propauxakdv}
If $w_0 \in L^2(0,L)$, $f\in L^1(0,T;L^2(0,L))$ and
$\vec{k}:=(k_1,k_2,k_3)\in {\cal K}_T $ with  ${\cal K}_T
=H^{-\frac13}(0,T) \times H^{\frac13}(0,T) \times L^2(0,T)$,  then
the system (\ref{2.4}) admits a unique solution $w \in X_T$ which,
in addition,  has the  hidden (or sharp trace) regularities
$$\partial _x^{j}w \in L_x^{\infty}(0,L, H^{\frac{1-j}{3}}(0,T))
\text{ for } j=0,1,2.$$ Moreover, there exist constants $C>0$,
$\,C_j>0$, $j=0,1,2$  such that \ben \label{auxakdv} \|w\|_{X_T} \le
C \left ( \|w_0\|_{L^2(0,L)} + \|\vec{k} \|_{{\cal K}_T}
+\|f\|_{L^1(0,T;L^2(0,L))}\right ), \een where,
$$\|\vec{k}\|_{{\cal K}_T} :=  \left(  \|k_1\|^2_{H^{-\frac13}(0,T)}+
\|k_2\|^2_{H^{\frac13}(0,T)}+ \|k_3\|^2_{L^2(0,T)} \right)^{1/2}$$
and
\[ \sup
_{0<x<L} \| \partial _x^{j} w (x,
 \cdot)\|_{H^{\frac{1-j}{3}} (0,T)} \leq C_j \left ( \|w_0\|_{L^2(0,L)} +
\|\vec{k} \|_{{\cal K}_T} +\|f\|_{L^1 (0,T;L^2(0,L))}\right )\] for
$j=0,1,2$.

\end{prop}

\noindent {\bf Proof.}
The proof is similar to the one in \cite{zhang-1}. Its sketch will be
presented in the Appendix for the convenience of the interested
readers.
$\blacksquare$\\

\medskip
Now we turn to prove Theorem  \ref{adprop1}.

\medskip
\noindent {\bf Proof of Theorem \ref{adprop1}}.  Let
$${\cal X}_T:= \big\{ u\in X_T; \quad \partial _x^{j}u\in
L^{\infty}_x(0,L; H^{\frac{1-j}{3}}(0,T)), \ j=0,1,2\big\}
$$
which is a Banach space equipped with the norm
\[ \| u\|_{{\cal X}_T} := \|u\|_{X_{T}} + \sum ^2_{j=0} \|
\partial _x^j u\|_{L^{\infty}_x(0,L; H^{\frac{1-j}{3}}(0,T))}.\]

According to Proposition
\ref{propauxakdv}, for any $v\in {\cal X}_{\beta} $ where $0<
\beta \leq T$, and any  $\va _0\in L^2 (0,L)$, the system
\begin{equation}\label{2.5}
\bc
w_t + w_{xxx} =-v_x, \quad (x,t)\in (0,L)\times (0,\beta),\\
w_{xx}(0,t)=-v(0, t), \quad  w(L,t)=0, \quad  w_x(L,t)=0, \quad t\in (0, \beta),\\
w(x,0)= \va _0(x), \quad x\in (0,L) \ec \end{equation} admits a
unique solution $w\in {\cal X}_{\beta }$ and, moreover,
\[ \|w\|_{{\cal X}_{\beta}} \leq C
\left ( \|\va _0 \|_{L^2 (0,L)} +\| v(0, \cdot )\|_{H^{-\frac13}
(0, \beta )} +\| v_x \|_{L^1 (0, \beta; L^2 (0,L))} \right ) \]
where the constant $C$ depends only on $T$. As we have
\[ \| v_x \|_{L^1 (0, \beta; L^2 (0,L))}\leq \beta ^{\frac12}
\|v\|_{{\cal X}_{\beta}}  \] and
\[ \|
v(0, \cdot )\|_{H^{-\frac13} (0, \beta )} \leq \| v(0, \cdot
)\|_{L^2  (0, \beta )}\leq \beta ^{\frac23} \| v(0, \cdot )\|_{L^6
(0, \beta)} \leq C\beta ^{\frac23 } \|v(0, \cdot)\|_{H^{\frac13} (0,
\beta)} \leq C \beta ^{\frac23} \| v\|_{{\cal X}_{\beta}} ,\]  the
system (\ref{2.5})  defines a map $\Gamma $ from the space ${\cal
X}_{\beta} $ to ${\cal X}_{\beta }$ for any $0< \beta \leq
\max\{1,T\} $ as follows
\[ w= \Gamma (v) \ \mbox{for any $v\in {\cal X}_{\beta}$} \]
where $w\in {\cal X}_{\beta} $ is the corresponding solution of
(\ref{2.5}) and
\[ \| \Gamma (v)\|_{{\cal X}_{\beta}} \leq C_1\|\va _0 \|_{L^2 (0,L)}
+C_2 \beta^{\frac12} \| v\|_{{\cal X}_{\beta}}\] where $C_1 $ and
$C_2$ are two constants depending only on $T$. Choose $r>0$ and
$0<\beta\leq \max\{1,T\} $ such that
\[ r=2C_1 \|\va _0\|_{L^2 (0,L)}, \quad 2C_2 \beta ^{\frac12} \leq \frac12.\] Then, for any
 $$v\in B_{\beta,
r} = \left\{ v \in {\cal X}_{\beta}; \quad \|v\|_{{\cal
X}_{\beta}}\le r\right\},$$ we have
\[ \| \Gamma (v)\|_{{\cal X}_{\beta }} \leq r. \]
Moreover, for any $v_1, \ v_2 \in B_{\beta, r}$, we get
\[ \| \Gamma (v_1)-\Gamma (v_2)\|_{{\cal X}_{\beta }}\leq 2C_2 \beta ^{\frac12} \|
v_1-v_2\|_{{\cal X}_{\beta}}\leq \frac12 \| v_1-v_2\|_{{\cal X}
_{\beta}} .\]
 Therefore the map $\Gamma$ is a contraction mapping on $B_{\beta,r}$. Its fixed point $w=\Gamma(w)\in
 {\cal X}_{\beta}$
 is the desired solution for $t\in (0,\beta)$.  As the chosen
 $\beta $ is independent of $\va _0$, the standard continuation
 extension argument yields that the solution $w$ belongs to ${\cal X}_T$.
 The proof is complete. $\blacksquare$

 \medskip
 Finally we conclude this  section  with an elementary estimate
 for solutions of system (\ref{2.3}).

\begin{prop} \label{bdness}
 Any solution $\varphi$ of the adjoint problem \eqref{2.3} with initial data
 $\varphi_0 \in L^2(0,L)$ satisfies
     $$\|\va_0\|_{L^2(0,L)}^2 \le \frac{1}{T}
     \|\va \|_{L^2((0,T)\times(0,L))}^2
     + \| \va_x(0, \cdot) \|_{L^2(0,T)}^2 + \| \va(0,\cdot)\|_{L^2(0,T)}^2.$$

 \end{prop}
\noindent {\bf Proof.}
Multiplying  both sides of the equation in (\ref{2.3}) by
$(T-t)\va$ and integrating by parts over $(0,L)\times (0,T)$, we get $$\izl (T-t) \va^2 \big|_0^T dx+\int_0^T (T-t)
\left(  \va^2(0,t)+ \va_x^2(0,t)  \right)dt
 +  \izt \izl  \varphi^2 dx dt  =0.$$
Consequently,
 \ben\label{q3}
    \izl \va_0^2 dx &\le&\frac{1}{T}\izl \izt \va^2 dt dx +  \izt \va^2(0,t)dt  +  \izt \va_x^2(0,t)
    dt.
 \een
  $\blacksquare$

\medskip
Equivalently, the following estimate holds for solutions $\psi $
of the system (\ref{10-1}): \be \label{z-1}\|\psi_T\|_{L^2(0,L)}^2
\le \frac{1}{T}
     \|\psi\|_{L^2((0,T)\times (0,L))}^2
     + \| \psi_x(L, \cdot) \|_{L^2(0,T)}^2 + \| \psi
     (L,\cdot)\|_{L^2(0,T)}^2.\ee
As a comparison, it is worth pointing out that for the adjoint
system of (\ref{2.2}), which is given by  \be\label{2.2-ad} \bc
\nu_t + \nu_x +  \nu_{xxx}  =0, \quad \nu(x,T)= \nu _T(x),\quad (x, t)\in (0,L)\times (0,T),\\
\nu (L,t)=0, \quad  \nu_x(0,t)=0, \quad \nu (0,t)=0,
 \ec\ee
 the following inequality holds
 \be\label{z-2}\|\nu_T\|_{L^2(0,L)}^2 \le \frac{1}{T}
     \|\nu\|_{L^2((0,T)\times (0,L))}^2
     + \| \nu_x(L, \cdot) \|_{L^2(0,T)}^2 .\ee
     The extra term $\| \psi
     (L,\cdot)\|_{L^2(0,T)}^2$ in (\ref{z-1})  brings new challenges in
     establishing the observability  of the adjoint system
     (\ref{10-1}).
\section{Linear control systems}\label{sec-lin}
\setcounter{equation}{0}
Consideration is first given to  boundary controllability of the
linear system
 \be \label{3-1}
\begin{cases}
y_t + y_x + y_{xxx} =0, \quad  (x,t)\in (0,L)\times (0,T),\\
y(0,t)=0, \, y_x(L,t)=h_2(t), \, y_{xx}(L,t)=0, \quad t\in (0,T),   \end{cases} \ee
which employs  only one control  input $h_2\in L^2(0,T)$.\\
\begin{prop}\label{pp}  Let   $L \notin
\FF$  (see \eqref{FF}) and $T>0$ be given. There exists a bounded linear operator
$$\Psi : L^2 (0,L)\times L^2 (0,L)\to L^2 (0,T) $$ such that  any
$y_0,\ y_T \in L^2 (0,L)$, if one chooses  $h_2=\Psi (y_0, y_T)$,
then system (\ref{3-1}) admits a solution $y\in X_T$
satisfying
\[ y|_{t=0}=y_0 , \qquad y|_{t=T} =y_T  .\]
\end{prop}
As it is well-known, the exact controllability  of the system
(\ref{3-1}) is related to the observability of its adjoint system
\be \label{3-2} \begin{cases}
\psi _t + \psi_x +  \psi_{xxx}  =0, \quad (x,t)\in (0,L)\times (0,T),\\
\psi(0,t)=0, \quad  \psi_x(0,t)=0, \quad   \psi(L,t)+\psi_{xx}(L,t)=0, \quad t\in (0,T),\\
\psi(x,T)= \psi_T(x), \quad x\in (0,L).  \end{cases} \ee

 \begin{lemm}\label{case3}
 For all $T>0$ and  all $L \notin  \FF$   there exists $C=C(L,T)>0$ such that  for any
  $  \psi_T\in L^2(0,L)$,
 the solution   $\psi$ of  \eqref{3-2} satisfies
 \be \label{3.3} \|\psi_T\|_{L^2(0,L)}\le C    \|\psi_x(L,t)\|_{L^2(0,T)}. \ee
 \end{lemm}
\noindent {\bf Proof.} Proceeding as in   \cite{97rosier}, if (\ref{3.3}) is false, then
there exists  a sequence $\{\psi^n_T\}_{n\in \NN}\in \lzl$ with
$\|\psi^n_T\|_{L^2 (0,L)} =1$ such that the corresponding solutions
of \eqref{3-2} satisfy
 $$1= \|\psi_T^n\|_{\lzl} > n \|\psi_x^n(L,\cdot )\|_{\lzt}.$$
Thus $\| \psi_x^n(L, \cdot ) \|_{\lzt} \to 0$ as $n\to \infty$.
  By Proposition \ref{fistpart} and Theorem \ref{adprop1}, the sequences  $\{ \psi^n\}_{n\in \NN}$  and $\{\psi ^n (L, t)\}_{n\in \NN}$ are  bounded in
   $L^2(0,T;H^1(0,L))$
  and $H^{\frac13} (0,T)$, respectively. In addition, according to
Proposition  \ref{bdness}
  \be\label{equ}\|\psi^n_T\|_{L^2(0,L)}^2 &\le& \frac{1}{T} \|\psi ^n\|_{L^2(0,T; L^2(0,L))}^2 +
   \| \psi^n_x(L, \cdot) \|_{L^2(0,T)}^2 + \| \psi^n(L, \cdot
   )\|_{L^2(0,T)}^2.
  \ee
Since, $\psi^n_t=-(\psi^n_x +\psi^n_{xxx})$ is bounded in
$L^2(0,T;H^{-2}(0,L)$ and by the embedding
$$H^1(0,L) \hookrightarrow L^2(0,L)\hookrightarrow  H^{-2}(0,L)$$
the sequence $\{\psi^n\}_{n\in \NN}$ is relatively compact in
$L^2(0,T;L^2(0,L))$ (see \cite{simon}). Furthermore,  the second term on the right in \eqref{equ}
converges to zero in $L^2(0,T)$, and by the compact embedding
$$H^{\frac13}(0,T)\hookrightarrow L^2(0,T)$$ the sequence
$\{\psi^n(L, \cdot)\}_{n\in \NN}$ has a convergent subsequence on $L^2(0,T)$.
Therefore
 $\{\psi^n_T\}_{n\in \NN}$ is a $L^2(0,L)$-Cauchy
sequence.
 Let us denote  $\psi_T=\lim_{n\to \infty}\psi_T^n$ and  $\psi$ be the corresponding solution of (3.2).
 Since
 $\psi_x^n(L, t)\to \psi _x(L,t)$ as $n\to \infty$ in $\lzt$ and  $\|\psi_T^n\|_{L^2 (0,L)}=1$ for any $n$,
 we have   $\| \psi_T\|_{\lzl}=1$ and
  $\psi_x(L,t)=0$.  By  the following Lemma \ref{next},  one can conclude that  $\psi\equiv0$,
  therefore $\psi_T(x)\equiv 0$ which  contradicts the fact that $\|\psi _T\|_{L^2 (0,L)}=1$.
  $\blacksquare$\\

\begin{lemm}\label{next} For given $T>0$, let us define
$$N_T=\{\psi_T\in \lzl :  \ \psi\in X_T \ \  \text{is the mild solution of}
\ (\ref{3-2}) \  \text{satisfying} \  \psi_x(L,\cdot)=0  \
\text{in} \ L^2 (0,T) \}.$$ Then, $N_T=\{ 0\}$ if and only if $L
\notin
  \FF$.
\end{lemm}
\noindent {\bf Proof.} The proof uses the same arguments as that given in \cite{97rosier} and will be presented in the Appendix  for
the convenience of the interested readers.
$\blacksquare$



\medskip
Now we turn to prove Proposition \ref{pp}.

\medskip
\noindent {\bf Proof of Proposition  \ref{pp}}. Without loss of
generality, we assume that $y_0=0$. Let $\psi $ be a solution
of the system (\ref{3-2}) and multiply both sides of the
equation in (\ref{3-1}) by $\psi$ and integrate over the domain
$(0,L)\times (0,T)$. Integration
by parts lead to
\[ \int ^L_0 y(x,T)\psi _t (x) dx = \int ^T_0 h_2 (t) \psi _x (L,t)
dt .\]
    Let us denote by   $\Upsilon$ the linear and bounded map from $\lzl \to \lzl $ defined by
$$\Upsilon : \psi_T(\cdot)  \to y(\cdot,T ) $$
 with $y$ being the solution of \eqref{3-1} when
 $h_2(t)=\psi_x(L,t)$ where $\psi$ is the solution of the system (\ref{3-2}). According to Lemma \ref{case3},
\beq \label{ope} (\Upsilon(\psi_T),\psi_T)_{L^2(0,L)}=
\|\psi_x(L,\cdot)\|_{L^2(0,T)}^2\ge C^{-2}
\|\psi_T\|_{L^2(0,L)}^2. \eeq Thus $\Upsilon$ is invertible by
Lax-Milgram Theorem. Consequently, for given $y_T\in L^2 (0,L)$, we can define $\psi _T = \Upsilon ^{-1} y_T$. We solve system (\ref{3-2}) and get $\psi \in X_T$. Then, we set  $h_2 (t) = \psi _x (L,t)$ in
system (\ref{3-1}) and see that the corresponding solution $y\in X_T$
satisfies
\[ y|_{t=0}=0, \quad y|_{t=T}= y_T  .\] The proof is complete.
$\blacksquare$

\medskip

Next  we turn to consider boundary controllability of the linear
system
 \be \label{3-6}
    \begin{cases}
y_t + y_x + y_{xxx}   =0, \quad   (x,t)\in (0,L)\times (0,T),\\
y(0,t)=h_1(t), \quad y_x(L,t)=h_2(t), \quad  y_{xx}(L,t)=0,\quad t\in (0,T),
\end{cases} \ee
 with two   control inputs  $h_1\in H^{\frac13} (0,T)$ and $h_2\in L^2 (0,T)$.
 \begin{prop}\label{two} Let $T>0$ be given. There exists a bounded linear
operator
$$\digamma :  L^2 (0,L)\times L^2 (0,L)\to H^{\frac13} (0,T)\times L^2 (0,T) $$ such that  for any
$y_0,\ y_T \in L^2 (0,L)$, if one chooses \[ (h_1, h_2 )=\digamma
(y_0, y_T),\] then  the system (\ref{3-6}) admits a solution $y\in
X_T$ satisfying
 \[ y|_{t=0}= y_0, \qquad y|_{t=T} =y_T .\]

 \end{prop}

 As before, we first establish the following observability
 estimate for the corresponding adjoint system (\ref{3-2}).

\begin{lemm}\label{ceA}
    Let  $T>0$ be given.  There exists  a constant $C>0$ such that for any $\psi_T\in L^2(0,L)$,
    the corresponding solution $\psi$ of  \eqref{3-2} satisfies
    \be\label{gf}
    \|\psi_T\|_{L^2(0,L)} \le C \left( \izt (|\Delta_t^{-\frac13}\psi_{xx}(0,t)|^2 + | \psi_x(L,t)|^2 )dt \right)
    \ee
    where $\Delta_t:=I-\partial_t^2$.
\end{lemm}
\noindent {\bf Proof.}
If the estimate (\ref{gf}) is false,    then there exists  a
sequence $\{\psi^n_T\}_{n\in\NN} \in L^2(0,L)$ with $\| \psi^n_T
\|_{L^2(0,L)}=1$ such that the corresponding solutions $\psi^n$ of
\eqref{3-2} satisfies
 $$1= \|\psi_T^n \|_{L^2(0,L)} > n \left( \izt ( |\Delta_t^{-\frac13}\psi_{xx}^n(0,t)|^2 + | \psi_x^n(L,t)|^2)
 dt \right)$$
 for any $n$. Thus
\be\label{cond1A} \|\Delta_t^{-\frac 1 3} \psi_{xx}^n(0,\cdot )
\|_{\lzt} \to 0 \quad \text{and} \quad \|\psi_{x}^n(L,\cdot )
\|_{\lzt} \to 0  \ee when $n\to \infty$. Arguing  as in the proof of
Lemma \ref{case3} we can conclude that $\{\psi_T^n\}_{n\in\NN}$ is a
Cauchy  sequence in $L^2(0,L)$ converging to some $\psi_T\in
L^2(0,L)$. The corresponding solution
 $\psi$ of (\ref{3-2}) satisfies $\psi_{xx}(0,t)=0$ and $\psi_x(L,t)=0$, i.e.,
\be\label{Aakdv}\left\{\ba{l}
\psi_t + \psi_x +  \psi_{xxx}  =0, \quad (x,t)\in (0,L)\times (0,T),\\
 \psi(0,t)=0, \,  \psi_x(0,t)=0,  \ \psi_{xx}(0,t)=0,\quad t\in (0,T),\\
 \psi(L,t)+\psi _{xx}(L,t)=0, \  \, \psi_{x}(L,t)=0,\quad t\in (0,T),\\
 \psi(x,T)= \psi_T(x),\quad x\in (0,L), \ea \right.\ee
 from which we have  $\psi\equiv 0$  because of the unique continuation property ($\psi(0,t)=\psi_{x}(0,t)=\psi_{xx}(0,t)=0$ for any $t\in (0,T)$).
 In particular, $\psi _T\equiv 0$
 which contradicts the fact that $\|\psi _T\|_{L^2 (0,L)} =1$.
$\blacksquare$

\medskip \noindent {\bf Proof of Proposition \ref{two}}: Without loss of generality,
we assume that $y_0=0$. Let $\psi $ be a solution of the
system (\ref{3-2}) and multiply both sides of the equation in
(\ref{3-6}) by $\psi $ and integrate over the domain $(0,L)\times
(0,T)$. Integration by parts leads
to

 $$
 \izl y(x,T)\psi _T(x) dx = \izt \left( h_1(t) \psi_{xx}(0,t) + h_2(t) \psi_x(L,t)\right
 )dt .$$

Let us denote by   $\Upsilon$ the linear and bounded map from $\lzl
\to \lzl $ defined by
$$\Upsilon : \psi_T(\cdot)  \to y(\cdot,T ) $$  with $y$ being the solution of \eqref{3-6} when
$$h_1(t)=\Delta_t^{-\frac13}\psi_{xx}(0,t)\quad \text{and} \quad
h_2(t)=\psi_x(L,t)$$
 where $\psi$ is the solution of system (\ref{3-2}). Thus
\[ (\Upsilon (\psi _T ), \psi _T)_{L^2 (0,L)} =(\Delta_t^{-\frac13} \psi_{xx}(0,\cdot), \Delta _t^{-\frac13}\psi
_{xx}(0,\cdot)_{L^2 (0,T)} +\| \psi_x (L,\cdot )\|^2_{L^2 (),L)}
\geq C^{-2}\| \psi _T \|^2_{L^2 (0,L)}.\] The proof is then
completed by using the Lax-Milgram Theorem. $\blacksquare$

\medskip
We  now consider boundary controllability of the linear system
 \be \label{3-7}
    \begin{cases}
y_t + y_x + y_{xxx}   =0, \quad   (x,t)\in (0,L)\times (0,T),\\
y(0,t)=0, \quad y_x(L,t)=h_2(t), \quad  y_{xx}(L,t)=h_3 (t),\quad t\in (0,T),

\end{cases} \ee
 with two   control inputs  $h_2\in L^2(0,T)$ and $h_3\in H^{-\frac 1 3}(0,T)$.
  \begin{prop}\label{two-2} Let $T>0$ be given. There exists a bounded linear
operator
$$\digamma _1 :  L^2 (0,L)\times L^2 (0,L)\to L^2 (0,T)\times H^{-\frac13}(0,T) $$ such that  for any
$y_0,\ y_T \in L^2 (0,L)$, if one chooses \[ (h_2, h_3 )=\digamma
_1 (y_0, y_T),\] then  the system (\ref{3-6}) admits a solution
$y\in X_T$ satisfying
 \[ y|_{t=0}= y_0, \qquad y|_{t=T} =y_T .\]

 \end{prop}
As before, Proposition \ref{two-2} follows from  the following
observability
 estimates for the corresponding adjoint system (\ref{3-2}).

\begin{lemm}\label{ceB}
    Let  $T>0$ be given.  There exists  a constant $C>0$ such that for any $\psi_T\in L^2(0,L)$,
    the corresponding solution $\psi$ of  \eqref{3-2} satisfies
    \be
    \|\psi_T\|_{L^2(0,L)} \le C \left( \izt (|\Delta_t^{\frac13}\psi (L,t)|^2 + | \psi_x(L,t)|^2 )dt .\right)
    \ee
\end{lemm}
\noindent {\bf Proof.} The proof is similar to that of Lemma \ref{ceA} and
is therefore omitted.
$\blacksquare$

\medskip
Finally we turn to consider the linear system associated to
(\ref{kdv2}) using only $g_1\in H^{\frac 1 3}(0,T)$ and $g_3\in L^2(0,T)$ as control inputs, i.e.
\be\label{kdv2-3-3}\left\{\ba{l}
u_t + u_x + u_{xxx}   =0,  \ (x,t)\in (0,L)\times (0,T), \\
u(0,t)=g_1 (t), \, u(L,t)=0, \, u_{x}(L,t)=g_3 (t),\quad t\in (0,T).
 \ea \right.\ee
 The critical length phenomenon will not  occur and system \eqref{kdv2-3-3} is exactly controllable for any $L>0$ as stated in the following result.
  \begin{prop}\label{two-2-2} Let $T>0$ be given. There exists a bounded linear
operator
$$\digamma _2:  L^2 (0,L)\times L^2 (0,L)\to H^{\frac13}(0,T) \times L^2 (0,T)  $$ such that  for any
$u_0,\ u_T \in L^2 (0,L)$, if one chooses \[ (g_1, g_3 )=\digamma
_2(u_0, u_T),\] then  the system (\ref{kdv2-3-3}) admits a
solution $y\in X_T$ satisfying
 \[ u|_{t=0}= u_0, \qquad u|_{t=T} =u_T .\]

 \end{prop}
 \noindent {\bf Proof.}  The proof is similar to that of Proposition
 \ref{two-2} and is therefore skipped.
$\blacksquare$

\section{Nonlinear control systems}\label{sec-non}

\setcounter{equation}{0}
 In this section we first   consider the
nonlinear system
\be \label{4-1}
\begin{cases}
y_x+y_{xxx}+y_{x} +yy_x=0, \quad x\in (0,L), \; t\in (0,T),\\
y(0,t)=0, \; y_x(L,t)=h_2 (t), \; y_{xx}(L,t)=0, \quad t\in (0,T),\\
y(x,0)=y_0 (x),\quad x\in (0,L),
\end{cases}
\ee
 and present the proof of Theorem \ref{theo-h-2}.

 \medskip
\noindent {\bf Proof of Theorem \ref{theo-h-2}.} Rewrite the
system (\ref{4-1}) in its integral form \beq\label{opfp}
 y(t)= W_0(t)y_0 + W_{bdr}(t)h_2 - \int_0^t W_0(t-\tau) (yy_x)(\tau) d\tau .
\eeq Here we have written $W_{bdr}(t)(0,h_2,0) $ as $W_{bdr}(t)h_2$
for simplicity. For  any $v\in X_T$, let us set
$$ \nu (T,v):= \izt W_0(T-\tau)(vv_x)(\tau) d\tau .$$
 For any $y_0, \ y_T\in
L^2(0,L)$, we use Proposition \ref{pp} to define
$$h_2 = \Psi (y_0,y_T + \nu (T,v)).$$  Then
$$v(t)=  W_0(t)y_0  + W_{bdr}\Psi(y_0,y_T + \nu(T,v))  - \int_0^t W_0(t-\tau) (vv_x)(\tau) d\tau$$
satisfies
\[ v|_{t=0}= y_0,\qquad
v|_{t=T} =y_T + \nu (T,v) -\nu (T,v)= y_T .\] This leads us to
consider the map
$$\Gamma (v)=  W_0(t)y_0  + W_{bdr}\Psi(y_0,y_T + \nu(T,v))  - \int_0^t W_0(t-\tau) (vv_x)(\tau,x) d\tau .$$
If we can show that the map $\Gamma $ is a contraction in an
appropriate metric  space, then its fixed point $v$ is a solution
of (\ref{4-1}) with $h_2 = \Psi (y_0, y_T+\nu (T; v))$ which
satisfies
\[ v|_{t=0}= y_0,\qquad
 v|_{t=T} =y_T  .\] Next we show that  this is indeed the case;
the map $\Gamma$ is a contraction map in the ball
\[ B_r=\{ z\in X_T; \  \| z\|_{ X_T}\leq r\} \]
for an appropriately chosen $r$.  According to Proposition \ref{allsys}, there
exists a constant
 $C_1>0$ such
that
\[ \|\Gamma (v)\|_{X_T} \leq C_1 \left (\| y_0\|_{L^2 (0,L)}
+\| \Psi (y_0,y_T+ \nu ( T,v))\|_{L^2 (0,L)} + \int ^T_0
\|vv_x\|_{L^2(0,L)}(t)  dt \right ).\]
 Since \[ \| \Psi (y_0, y_T+ \nu (T,v))\|_{L^2 (0,L)}\leq C_2 \left
( \|y_0\|_{L^2 (0,L)} +\|y_T\|_{L^2 (0,L)} +\|\nu (T,v)\|_{L^2
(0,L)} \right ),
\]
\[ \|\nu (T,v)\|_{L^2 (0,L)}\leq \int ^T_0 \| W_0(T-\tau ) vv_x \|_{L^2
(0,L)} d\tau \leq \int ^T_0 \| vv_x\|_{L^2 (0,L)} (t) dt, \] and the bilinear estimate
\[ \int ^T_0 \|vv_x\|_{L^2 (0,L)}(t)  dt\leq C_3 \|v\|_{X_T}^2 \]
we arrive at
\[ \|\Gamma (v)\|_{X_T} \leq C_3 (\|y_0\|_{L^2 (0,L)} +\|y_T\|_{L^2
(0,L)} ) + C_4 \|v\|_{X_T}^2 \] for any $v\in X_T$ where $C_3 $
and $C_4 $ are constants depending only on $T$. By choosing $r,\delta$ such that
\be
\label{contract} r=2C_3 \delta, \qquad  4C_3C_4\delta < \frac12
\ee we get
\[ \|\Gamma (v)\|_{X_T} \leq C_3\delta +4C_4 C_3\delta  C_3 \delta
\leq 2 C_3 \delta \leq r \] for any $v\in B_r$. In addition, for
$v_1, \ v_2 \in B_r$,
\[
 \Gamma (v_1)-\Gamma (v_2) =W_{bdr}\Psi \left (0, \nu (T, v_1)-\nu (T, v_2)\right ) +\frac12\int
^t_0 W_0(t-\tau) \left [(v_1+v_2)(v_1-v_2)_x \right ](\tau ) d\tau
\] and
\begin{eqnarray*}
\| \Gamma (v_1)-\Gamma (v_2)\|_{X_T}  &\leq  & \left \{ C_4
(\|v_1\|_{X_T} +\|v_2\|_{X_T} )  +C_4 (\| v_1\|_{X_T} +\|v
_2\|_{X_T}) \right \} \| v_1-v_2\|_{X_T} \\  &\leq & 8C_3C_4
\delta \| v_1-v_2\|_{X_T}\\ &\leq  & \alpha \| v_1-v_2 \|_{X_T}
\end{eqnarray*}
 with $\alpha = 8C_3 C_4 \delta <1 .$ The proof is  completed. $\blacksquare$

 \medskip
 As Theorem \ref{theo-h-1-2} and  Theorem \ref{theo-h-2-3} can be
 proved using the same arguments as those in the proof of Theorem
 \ref{theo-h-2}, their proofs will be skipped.

 Now we
 turn to consider the system
 \be \label{4-11}
\begin{cases}
y_x+y_{xxx}+y_{x} +yy_x=0, \quad x\in (0,L), \; t\in (0,T),\\
y(0,t)=0, \; y_x(L,t)=0, \; y_{xx}(L,t)=h_3(t),\quad t\in (0,T)\\
\end{cases}
\ee to prove Theorem \ref{theo-h-3}.

 \medskip
 \noindent
 {\bf Proof of Theorem \ref{theo-h-3}.} We first consider the
 system
 \be\label{6.1}\left\{\ba{l}
u_t + u_x + u_{xxx} + u u_x =0,  \ (x,t)\in (0,L)\times (0,T), \\
u(0,t)=0, \, u(L,t)=g_2(t), \, u_{x}(L,t)=0, \quad t\in (0,T),
 \ea \right.\ee
and show that the following controllability result holds, which is
an improvement of Theorem C due to Glass and Guerrero \cite{09glass-guerrero}.

\medskip
\begin{prop}\label{6:1} Let $T>0$ and $L\notin {\cal N}$ be given. There
exists  $r>0$ such that for any $u_0, \ u_T \in L^2 (0,L)$ with
\[ \|u_0 \|_{L^2 (0,L)}+\|u_T \|_{L^2 (0,L)} \leq r,\] there
exists  $g_2\in H^{\frac13} (0,T) $ such that the system
(\ref{6.1}) admits a solution
\[ u\in C([0,T]; L^2 (0,L))\cap L^2 (0,T; H^1 (0,L))\]
satisfying
\[ u|_{t=0}= u_0, \qquad u|_{t=T} =u_T.\]
\end{prop}
If Proposition \ref{6:1} holds,  for given $y_0, \ y_T\in L^2
(0,L)$, set
\[ u_0=y_0, \quad u_T=y_T. \]
Then by Proposition \ref{6:1}, there exists $g_2\in
H^{\frac13}(0,T)$ such that  (\ref{6.1}) admits a unique solution
$u\in X_T$ satisfying
\[ u|_{t=0}= y_0, \qquad u|_{t=T} =y_T.\] Thus $y(x,t):=u(x,t)$
will be a desired solution of (\ref{4-11}) with $h_3 (t)=
u_{xx}(L,t)$ satisfying
\[ y|_{t=0}=y_0, \quad y|_{t=T}=y_T. \]
As $u\in X_T$ solves
\[
\left \{\begin{array}{l}  u_t+u_x +u_{xxx}=f, \quad u(x,0)=y_0,
\quad (x,y)\in (0,L)\times (0,T), \\ u(0,t)=0, \quad u(L,t)=
g_2(t), \quad u_x (L,t)=0,\quad t\in (0,T),  \end{array} \right. \] with $g_2\in
H^{\frac13}(0,T)$ and $f=-uu_x \in L^1 (0,T; L^2 (0,L))$, it
follows from Proposition \ref{regular} that  $u_{xx} (L, \cdot)\in
H^{-\frac13}(0,T)$. Thus, it suffices to prove Proposition
\ref{6:1} to complete the proof of Theorem \ref{theo-h-3}.

\medskip
To this end, note that, according to Theorem C,  $g_2\in
H^{\frac16-\epsilon } (0,T)$. We just need to prove that this
$g_2$ given by Theorem C belongs, in fact,  to the space $
H^{\frac13} (0,T)$. Indeed, the solution $u\in X_T$ given in
Theorem C can be written as
\[ u= \kappa + \mu  \]
where $\kappa$ solves
\[ \left \{ \begin{array}{l} \kappa_t +\kappa _{xxx}=f, \quad \kappa
(x,0)=u_0 , \ (x,t)\in (0,L)\times (0,T), \\ \kappa (0,t)=\kappa
(L,t) =\kappa _x (L,t)=0, \quad t\in (0,T),  \end{array} \right. \] with
$f=-u_x -uu_x$, and  $\mu$ solves \be
\label{iff} \left \{
\begin{array}{l} \mu_t +\mu _{xxx}=0, \quad \mu (x,0)=0 , \
(x,t)\in (0,L)\times (0,T), \\ \mu (0,t)=0, \quad \mu (L,t)
=g_2(t), \quad  \mu  _x (L,t)=0, \quad t\in (0,T). \end{array} \right. \ee   As $u_0\in L^2 (0,L)$ and  $f=-u_x -uu_x \in L^1
(0,T; L^2 (0,L))$ (because $u\in X_T$), we have $\kappa \in X_T$ by
Proposition \ref{regular}. In addition, the following lemma (whose proof will be presented in the Appendix) holds for system (\ref{iff}).

\begin{lemm}\label{iff-1} The solution $\mu$ of  (\ref{iff}) belongs to $X_T$  if
and only if $g_2$ belongs to $H^{\frac13} (0,T)$.
\end{lemm}

Using this Lemma, we get $u\in X_T$ if and only if $g_2\in H^{\frac13} (0,T)$.
The proof of
Theorem \ref{theo-h-3} is  complete. $\blacksquare$

\medskip
Finally we consider the system \be \label{4-7}
\begin{cases}
y_x+y_{xxx}+y_{x} +yy_x=0, \quad x\in (0,L), \; t\in (0,T),\\
y(0,t)=h_1(t), \; y_x(L,t)=h_2 (t), \; y_{xx}(L,t)=h_3 (t),\quad t \in (0,T),\\
y(x,0)=y_0 (x),\quad x \in (0,L),
\end{cases}
\ee and present the proof of Theorem \ref{theo-h-1-2-3}.

\medskip
\noindent {\bf Proof of Theorem \ref{theo-h-1-2-3}.} Consider
first the following initial value control problem for the KdV
equation posed on the whole line $\R$ \be \label{line}
\begin{cases} z_t+z_x+zz_x+z_{xxx}=0, \quad x, \ t\in \R, \\
z(x,0)= h(x) \end{cases} \ee where the initial value $h(x)$ is
considered as a control input. The following result is due to Zhang
\cite{Za}.

\medskip
\noindent {\bf Theorem G.}  Let $s\geq 0$ and $T>0$ be given and
suppose $w\in C(\R; H^{\infty} (\R))$ is a given solution of
\[ w_t+w_x+ww_x+w_{xxx}=0, \quad x, \, t\in \R .\]
There exists  $\delta >0$ such that for any $y_0, \ y_T \in H^s
(0,L)$ with \be \label{line-1} \|y_0-w(\cdot, 0)\|_{H^s (0,L)}\leq
\delta, \qquad \|y_T-w(\cdot, T)\|_{H^s (0, L)} \leq \delta ,\ee one
can find a control input $h\in H^s (\R)$ which is an external modification of $y_0$ such that (\ref{line})
admits a solution $z\in C(\R; H^s(\R))$ satisfying
\[ z(x,0)=y_0(x), \qquad z(x,T)=y_T(x) \quad for \ any \ x\in
(0,L) .\]

\medskip
Let $u$ be as in Theorem \ref{theo-h-1-2-3}. Applying Theorem G with
$w=u$ and $s=0$, we get the existence of $\delta>0$ such that for
any $y_0, \ y_T \in L^2  (0,L)$ with \be \label{line-6}
\|y_0-u(\cdot, 0)\|_{L^2 (0,L)}\leq \delta, \qquad \|y_T-u(\cdot,
T)\|_{L^2  (0, L)} \leq \delta ,\ee there exists $h\in L^2  (\R)$
and the corresponding
 $z\in C(\R; L^2 (\R))\cap L^2 (\R;
H^1_{loc} (\R))$  solution of (\ref{line}).
Let $y$ be the restriction of $z$ to the domain
$(0,L)\times (0,T)$, and
\[ h_1 (t)= z(0,t), \quad h_2 (t)=z_x (L,t), \quad h_{3} (t)=
z_{xx} (L,t) \] for $0<t<T$.  Then, according to \cite{zhang-1}, we
have that $y\in X_T$ solves (\ref{4-7}) with $h_1\in H^{\frac13} (0,T),
\ h_2 \in L^2 (0,T) $ and $h_3\in H^{-\frac13} (0,T)$.  Moreover,
\[ y|_{t=0}=y_0, \qquad y|_{t=T} = y_T .\]
The proof is complete. $\blacksquare$

\section{Conclusion Remarks }\label{conclusion}
\setcounter{equation}{0}
The focus of  our discussion has been
on the boundary controllability of two classes of boundary control
systems described by the KdV equation posed on a finite domain
$(0,L)$, namely,
\be \label{7.1}
\begin{cases} u_t+u_x +uu_x +u_{xxx}=0, \quad (x,t)\in (0,L)\times
(0,T), \\ u(0,t)=g_1(t), \quad u(L,t)=g_2 (t), \quad u_x(L,t)= g_3
(t), \quad t\in (0,T), \end{cases} \ee
 and
\be \label{7.2}
\begin{cases} y_t+y_x +yy_x +y_{xxx}=0, \quad (x,t)\in (0,L)\times
(0,T), \\ y(0,t)=h_1(t), \quad y_x(L,t)=h_2 (t), \quad y_{xx}(L,t)=
h_3 (t)\quad t\in (0,T). \end{cases} \ee The linear systems
associated to these equations are obtained by dropping the nonlinear
term $uu_x$ and $yy_x$, respectively.  The  system (\ref{7.1}) has
been intensively studied and various controllability results have
been established in the past. However, there have been  few results
for the second system (\ref{7.2}) because of  some difficulties to
apply directly  the methods that work effectively for the system
(\ref{7.1}). In this paper, aided by the newly established   hidden
regularities of solutions of the KdV equation, we have succeeded in
overcoming those difficulties and established various boundary
controllability results for the system (\ref{7.2}) similar to those
known for the system (\ref{7.1}) in the literature. Furthermore,
with the new tool in hand, we have also  be able to improve some
known controllability results for the system (\ref{7.1}). Our
results can be summarized as below.

\begin{description}
    \item[(i)]  The linear system associated to \eqref{7.2}   is exactly controllable
    with two or three boundary  controls in action.  In any of those cases, the nonlinear system \eqref{7.2}
    is locally exactly controllable.
    \item[(ii)]   With only a single control $h_2$ in action ($h_1=h_3=0$), the linear system associated to \eqref{7.2}
     is exactly controllable if and only if $L$ does not belong to
     \ben
\FF = \left\{ L\in \RR^{+}:  L^2=-(a^2+ab+b^2)  \ \text{with} \ a,b
\in \CC \ \text{satisfying}   \quad \frac{e^a}{a^2}= \frac{e^b}{b^2}
=\frac{e^{-(a+b)}}{(a+b)^2}\right\}. \een
The nonlinear system
\eqref{7.2} is  also locally exactly controllable if $L\notin \FF$.
\item[(iii)]  When  only  control input $h_3$ is employed ($h_1=h_2=0$), the linear system associated to \eqref{7.2}
 is exactly controllable if and only if $L$ does not belong to
  \ben
  {\cal N}  =  \Big\{ L\in \RR^+:
L^2=-(a^2+ab+b^2) \ \text{with} \ a,b \in \CC \nonumber
 \ \text{satisfying}
  \quad
 ae^a  =   be^b
= -(a+b) e^{-(a+b)} \Big\}.\een Moreover, if $L\notin \cal N$, then
the nonlinear system \eqref{7.2} is locally exactly controllable.
\item[(iv)] The linear system associated to \eqref{7.1}   is exactly controllable
    with control inputs $g_1$ and $g_2$  in action (put $g_3=0$).  In this case, the nonlinear system \eqref{7.1}
    is locally exactly controllable.
\item[(v)] We have improved the regularity of the control $g_2$ in \eqref{7.1}. In previous work \cite{09glass-guerrero},
the control $g_2$ is known to belong to the space $H^{\frac16
-\epsilon}(0,T)$ for any $\epsilon>0$. In this paper we are able to
prove that the control input $g_2$ belongs in fact to the space
$H^{\frac 1 3}(0,T)$.
\end{description}

While some significant progresses have been made in the study of
boundary controllability of the KdV equation on a bounded domain,
there are still a lot of interesting  questions left open for
further investigations. One of them is the so-called critical length
problem. As it is well known now,  the linear systems associated to
(\ref{7.1}) and (\ref{7.2}) are not always exactly controllable if
only a single control input is allowed to act on the right end of
the spatial domain $(0,L)$. In general, if the associated linear
system is not exactly controllable, one would intend to believe the
nonlinear system is also not exactly controllable. However, for the
system (\ref{7.1}) with only control input $g_3$ in action, though
its associated linear system is not exactly controllable when $L\in
S$ (see \eqref{NN-1} for the definition of $S$),  the nonlinear
system (\ref{7.1}) has been  shown by  Coron and Crepeau
\cite{04coron-crepeau}, Cerpa \cite{07cerpa}, and Cerpa and Creapeau
\cite{09cerpa-crepeau} to be locally (large time) exactly
controllable. The questions still
remain open for other critical length problems.\\

\begin{open}(Critical length problems)
\begin{itemize}
\item[(a)] Is the nonlinear system (\ref{7.2}) with only control
input $h_2$ in action exactly controllable  when the length $L$
of the spatial domain $(0,L)$ belongs  to the set $\FF$?

\item[(b)] Is the nonlinear system (\ref{7.2}) with only control
input $h_3$ in action  exactly controllable  when the length $L$
of the spatial domain $(0,L)$ belongs  to the set $\cal N$?

\item[(c)]  Is the nonlinear
system (\ref{7.1}) with only control input $g_3$ in action  exactly
controllable  when the length $L$ of the spatial domain $(0,L)$
belongs  to the set $\cal N$?

\end{itemize}
\end{open}

Most controllability results that have been established so far for both
systems (\ref{7.1}) and (\ref{7.2}) are local: one can only guide a
small amplitude initial state to a small amplitude terminal state by
choosing appropriate boundary control inputs. The following question
arises naturally.

\begin{open}\label{op2} (Global controllability problem)
Are the nonlinear systems (\ref{7.1}) and (\ref{7.2}) globally
exactly boundary controllable?
\end{open}

The following global interior stabilization result for the KdV
equation on a finite interval

\be \label{7.3}
\begin{cases} u_t+u_x +uu_x +u_{xxx} +a(x) u=0, \quad u(x,0)=u_0 (x), \quad  x\in (0,L),\,t>0,
\\ u(0,t)=0, \quad u(L,t)=0, \quad u_x(L,t)= 0, \quad t>0 \end{cases} \ee is well-known in the literature (see
\cite{perla, pazoto, rosier-z}).

\medskip
\noindent {\bf Theorem G.} {\em Assume the function $a\in
L^{\infty}(0,L)$ with $a(x)\geq 0$ and such that the support of $a$ is a nonempty
open subset of $(0,L)$. There exists $\gamma >0$ such that for any
$u_0 \in L^2 (0,L)$, the corresponding solution $u$ of (\ref{7.3})
belongs to the space $C([0, \infty); L^2 (0,L))$ and, moreover,
\[ \| u(\cdot, t)\|_{L^2 (0,L)}\leq \alpha (\|u_0 \|_{L^2 (0,L)})
e^{-\gamma t} \qquad \forall \ t\geq 0\] where $\alpha : \R^+\to \R
^+$ is a nondecreasing continuous function.}

 \vspace{.1in}
 \noindent
Combining Theorem G, Theorems \ref{theo-h-3},\ref{theo-h-1-2},
\ref{theo-h-2-3}  we have the following partial answers to the Open
Problem \ref{op2} for the nonlinear system (\ref{7.2}).
\begin{theo}\label{theo-5-1} There exists  $\delta >0$. For any $N>0$, one can find a $T>0$
depending only on $N$  and $\delta $ such that for any
$\phi , \ \psi \in L^2 (0,L)$ with $$\| \phi \|_{^2 (0,L)}\leq N,
\qquad  \| \psi \|_{^2 (0,L)}\leq \delta, $$ one can find either
\[ h_1\in H^{\frac13} (0,T), \quad h_2\in L^2 (0,T), \quad h_3=0,\] or
\[ h_1\in H^{\frac13} (0,T), \quad h_2=0,\quad h_3\in H^{-\frac13} (0,T),
\] or
\[ h_1=0,\quad h_2\in L^2  (0,T), \quad h_3\in H^{-\frac13} (0,T), \]
such that the nonlinear system (\ref{7.2}) admits a solution
$u\in C([0,T]; L^2 (0,L))$ satisfying
\[ u|_{t=0}= \phi , \qquad u|_{t=T} = \psi .\]
\end{theo}
\begin{proof} Note first that for any $\phi \in L^2 (0,L)$, the
system (\ref{7.3}) admits a unique solution $u\in C([0,\infty); L^2
(0,L))$ which also possesses   the hidden regularity
\[ \partial ^k_x u(x,t) \in L^{\infty}_x (0,L;
H^{\frac{1-k}{3}}(0,T), \quad k=0, 1, 2.\] Then Theorem \ref{theo-5-1} follows
from Theorem G and Theorems \ref{theo-h-3}, \ref{theo-h-1-2}, \ref{theo-h-2-3} using the same argument as
that used in the proof of Theorem 3.22 in \cite{Za-survey}.
\end{proof}

\begin{rem}
 Theorem 5.1 only provides a partial answer to Problem 5.2 since   the amplitude of
 the terminal state is still required
 to be small. Question  remains:

 \smallskip
 Can small amplitude restriction on the terminal state be removed?
 \end{rem}

If one is allowed to use all three boundary control inputs,
 then the small amplitude restriction can be removed.

 \begin{theo} Let $N>0$ be given. There exists a $T>0$ such that for any $\phi , \
\psi \in L^2 (0,L)$ with $$\| \phi \|_{L^2 (0,L)}\leq N, \qquad \|
\psi \|_{L^2 (0,L)}\leq N, $$   the nonlinear equation \[ u_t +u_x
+uu_x +u_{xxx}=0, \quad x\in (0,L)\times (0,T) \] admits a solution
$u\in C([0,T]; L^2 (0,L))$ satisfying
\[ u|_{t=0}= \phi , \qquad u|_{t=T} = \psi .\]
 \end{theo}
 \begin{proof} For given $\phi, \ \psi \in L^2 (0,L)$, let
 $\tilde{\phi}$ and $\tilde{\psi}$ be their extension from $(0,L)$
 to $(0,2L)$ such that
 \[  \tilde{\phi} \in L^2 (0,2L), \qquad \tilde{\psi}\in L^2 (0,2L),
 \qquad \int ^{2L}_0 \tilde{\phi}(x)dx = \int ^{2L}_0
 \tilde{\psi}(x)dx \]
 and consider the following internal control problem of the KdV
 equation posed on the interval $(0,2L)$ with periodic  boundary
 condition
 \[ \begin{cases}
 v_t+v_x+vv_x +v_{xxx}+ a(x) v=0, \quad v(x,0)=\phi (x), \qquad x\in (0,2L), \\
 v (0,t)=v(2L,t), \qquad v_x (0,t)= v_x (2L,t), \qquad
 v_{xx}(0,t)=v_{xx} (2L,t) \end{cases}\]
 where $a\in L^{\infty} (0,2L)$ and support of $a\in  (L,2L)$. The
 proof is completed by invoking Theorem 1.1 in \cite{lrz}.
 \end{proof}

\medskip
 Consequently, if chooses
 \[ h_1(t) =u(0,t), \qquad h_2(t) =u_x (L,t), \qquad
 h_3 (t)=u_{xx}(L,t), \]
 then the system (\ref{7.2}) will be guided from the given initial
 state $\phi $ to the given terminal state $\psi$. The only
 drawback is that we do not know exactly the regularities  of the
 boundary inputs $h_j$, $j=1,2,3$.

\medskip
 In Theorem \ref{theo-5-1}  and Theorem 5.3 the time interval $(0,T)$ used to conduct
control depends on the size of the initial state and terminal state. The larger of the size of the initial state,  the longer the
time interval $(0,T)$. Such type of controllability is usually
called the large time controllability. As it is well-known, the KdV
equation possesses infinite propagation speed. Thus one may wonder the following.

\begin{open}
Can the time interval $(0,T)$ in Theorems \ref{theo-5-1} and 5.3 be
chosen arbitrarily small?
\end{open}


\section{Appendices}
\setcounter{equation}{0}
\subsection{Proofs of Proposition \ref{propauxakdv} and Lemma \ref{iff-1}}

\medskip
\noindent {\bf Proof of Proposition \ref{propauxakdv}.} \quad The
solution of the system \be\label{newlkdv} \left\{\ba{l}
w_t + w_{xxx}  =f, \quad w(x,0)= w_0(x), \quad  (x, t) \in (0,L)\times (0,T),\\
w_{xx}(0,t)=k_1 (t), \quad  w(L,t)=k_2 (t), \quad  w_x(L,t)=k_3
(t), \quad t\in (0,T)
  \ea \right.\ee  can be written as
$$w(t) = W_0(t)w_0  + W_{bdr}(t) \vec{k} + \int_0^t W_0(t-\tau ) f(\tau) d\tau $$
with $\vec{k}= (k_1, k_2, k_3 )$ where $W_0(t)$ is the $C_0$
semigroup in $L^2 (0,L)$ generated by the operator \[ Bf= -f'''\]
with the domain \[ {\cal D}(B)= \{ f\in H^3(0,L); \
f''(0)=f(L)=f'(L)=0 \} .\]  Then, $u(t)=W_0 (t) w_0 $ solves
\be\label{4.1} \left\{\ba{l}
u_t +u_{xxx}  =0, \quad u(x,0)= w_0(x), \quad (x,t)\in (0,L)\times (0,T),\\
u_{xx}(0,t)=0, \quad  u(L,t)=0, \quad  u_x(L,t)=0 ,\quad t\in (0,T),
  \ea \right.\ee $v(t)=W_{bdr} (t) \vec{k}$ solves
\be\label{4.-2} \left\{\ba{l}
v_t + v_{xxx}  =0, \quad v(x,0) = 0, \quad  (x,t)\in (0,L)\times (0,T),\\
v_{xx}(0,t)=k_1 (t), \quad  v(L,t)=k_2 (t), \quad  v_x(L,t)=k_3 (t),\quad t\in (0,T),\\
  \ea \right.\ee and $z(t)=\int ^t_0 W_0(t-\tau )
f(\tau) d\tau $ solves \be\label{4.3} \left\{\ba{l}
z_t + z_{xxx}  =f, \quad z(x,0) = 0 , \quad  (x,t)\in (0,L)\times (0,T),\\
z_{xx}(0,t)=0, \quad  z(L,t)=0, \quad  z_x(L,t)=0, \quad t\in (0,T).
 \ea \right.\ee As in the  proof of Proposition \ref{allsys}, it is easy to
see that for any $f\in L^1 (0,T; L^2 (0,L))$ and $w_0 \in L^2
(0,L)$, both $u=W_0 (t)w_0 $ and $z=\int ^t_0 W_0 (t-\tau )
f(\tau)d \tau$ belong to the space $X_T$ and, in addition, there
exists a constant $C>0$ such that
\[ \| u\|_{X_T}+\|z\|_{X_T} \leq C\left (\|w_0\|_{L^2 (0,L)} +\|
f\|_{L^1 (0,T; L^2 (0,L))} \right ).\]

\medskip

For $v(t)=W_{bdr} (t)\vec{k}$,  following \cite{bsz-finite}, we
first look for an explicit representation formula. Applying  the
Laplace transform with respect to $t$ in both sides of the equation
in (\ref{4.-2}),
 (i.e. $\h{v}(s,x)=\int_0^{\infty} e^{-st} v(t) dt$), we obtain
\be\label{kdvlp} \bc
s\h{v} + \h{v}_{xxx}  =0, \quad x\in (0,L), \ s>0,\\
\h{v}_{xx}(0,s)=\h{k}_1(s), \quad  \h{v}(L,s)=\h{k}_2(s), \quad  \h{v}_x(L,s)=\h{k}_3(s), \quad  s>0.\\
 \ec\ee
Its solution $\h{v}(x,s)$   can be written  as $\h{v}(x,s)=
\sum_{j=1}^3 c_j(s) e^{\l_j (s)x}$ where $\l _j$ solves
characteristic equation $ s+\l^3=0$, i.e.  \ben \l_1 &=& i\rho,
\quad
       \l_2= -i\rho \left( \frac{1+i\sqrt{3}}{2}\right), \quad
    \l_3= -i\rho \left( \frac{1-i\sqrt{3}}{2}\right)
\een with  $s= \rho^3$.  Imposition of  the boundary conditions of
(\ref{kdvlp}) yields that $c_j=c_j(s)$ for $j=1,2,3$ solves the
system

\[
\left(
\begin{array}{ccc}
\l_1^2  &\l_2^2   &\l_3^2   \\
e^{\l_1 L}  & e^{\l_2 L}  & e^{\l_3 L}   \\
\l_1 e^{\l_1 L}  & \l_2 e^{\l_2 L}  &  \l_3 e^{\l_3 L}
\end{array}
\right)
\left(
\begin{array}{ccc}
  c_1 \\
  c_2 \\
  c_3
\end{array}
\right)
=\left(
\begin{array}{ccc}
  \h{k}_1 \\
  \h{k}_2 \\
  \h{k}_3
\end{array}
\right).\] By Cramer rule,  $$c_j =\frac{\Delta_j}{\Delta}, \quad
\text{for} \quad j=1,2,3,$$ where \[ \Delta =\Delta (s) = \left |
\begin{array}{ccc}
\l_1^2  &\l_2^2   &\l_3^2   \\
e^{\l_1 L}  & e^{\l_2 L}  & e^{\l_3 L}   \\
\l_1 e^{\l_1 L}  & \l_2 e^{\l_2 L}  &  \l_3 e^{\l_3 L}
\end{array}
\right | \]
 and $\Delta _j (s) $ is  the determinant of the matrices obtained
by changing the $j$th-column of $\Delta$ by the vector
 $(\h{k}_1,\h{k}_2,\h{k}_3)^T$ for $j=1,2,3$.
Taking the inverse Laplace transform of $\widehat{v}$ and
following the same arguments as those in
   \cite{03bona-sun-zhang} lead us to  the following representation of the
  solution $v$ of the system (\ref{4.-2}):
 \[ v(x,t)=\sum ^3_{m=1} v_m (x,t)\]
 with
 \[ v_m (x,t)=\sum ^3_{j=1}v_{j,m}(x,t) \quad \text{and} \quad
  v_{j,m}(x,t)= v_{j,m}^+ (x,t)+v_{j,m}^-(x,t)\]
 where  for $m, \ j=1,2,3$,
\[ v_{j,m}^+ (x,t)= \displaystyle\frac{1}{2\pi}   \int^{\infty}_0  e^{i\rho^3t+ \l^+_j(\rho)x}
    \frac{\Delta_{j,m}^+(\rho)}{\Delta ^+(\rho)} \h{k}_m^+(\rho)
    3\rho^2 d\rho,\]
     \[ v_{j,m}^-(x,t)= \overline{v_{j,m}^+ (x,t)}\]
    and \[\h{k}_m^+(\rho)=\h{k}_m(i\rho ^3), \ \Delta^+(\rho)=\Delta (i\rho ^3), \ \Delta ^+_{j,m}
    (\rho)= \Delta_{j,m} (i\rho ^3),\
    \lambda _j ^+ (\rho)=\lambda _j (i\rho^3).\]

\begin{lemm}\label{e-1} Let $T>0 $ be given. There exists a constant $C>0$ such that for any
$\vec{k}\in {\cal K}_T$, the system  (\ref{4.-2}) admits a unique
solution $v\in X_T$. Moreover, there exists a constant $C>0$
such that
\[ \|v\|_{X_T} +\sum _{j=0}^2 \|\partial _x^j v\|_{L^{\infty}_x
(0,L; H^{(1-j)/3}(0,T))} \leq C\|\vec{k}\|_{{\cal K}_T} .\]
\end{lemm}

\noindent {\bf Proof.}
Note that as stated above, the solution $v$ can be written as $$v(x,t)=v_1(x,t)+v_2(x,t)+v_3(x,t).$$ Let us prove Lemma  \ref{e-1}  for  $v_1$.   First of all, by straightforward
computation, we can list
  the asymptotic behavior  of the ratios $\frac{\Delta_{j,m}^+(\rho)}{\Delta^+(\rho)}$
   for $\rho\to + \infty$
as below.

\begin{center}
\renewcommand{\arraystretch}{2}
\begin{tabular}{||>{$}c<{$}| >{$}c<{$}|>{$}c<{$}||}\hline
\frac{\Delta_{1,1}^+(\rho)}{\Delta ^+(\rho)}  \sim \rho^{-2}e^{-\frac{\sqrt{3}}{2}\rho L}
 & \frac{\Delta_{2,1}^+(\rho)}{\Delta
^+(\rho)}
 \sim \rho^{-2} e^{-\sqrt{3}\rho L} & \frac{\Delta_{3,1}^+(\rho)}{\Delta ^+(\rho)}
 \sim \rho^{-2} e^{-\sqrt{3}\rho L} \\
\hline \frac{\Delta_{1,2}^+(\rho)}{\Delta ^+(\rho)} \sim 1 &
\frac{\Delta_{2,2}^+(\rho)}{\Delta ^+(\rho)} \sim
e^{-\sqrt{3} \rho L} & \frac{\Delta_{3,2}^+(\rho)}{\Delta ^+(\rho)}  \sim 1 \\
 \hline
\frac{\Delta_{1,3}^+(\rho)}{\Delta ^+(\rho)}  \sim  \rho^{-1}
&\frac{\Delta_{2,3}^+(\rho)}{\Delta ^+(\rho)}  \sim
 \rho^{-1} e^{-\frac{\sqrt{3}}{2}\rho L} & \frac{\Delta_{3,3}^+(\rho)}{\Delta ^+(\rho)}  \sim  \rho^{-1}
   \\ \hline
\end{tabular}
\end{center}

\medskip

\noindent As
\[ v_1(x,t)= \frac{3}{\pi} \sum ^3_{j=1}\displaystyle  {\cal R}e\  \int^{\infty}_0
e^{i\rho^3t}e^{ \l^+_j(\rho)x}
    \frac{\Delta_{j,1}^+(\rho)}{\Delta ^+(\rho)} \h{k}_1^+(\rho)
     \rho^2 d\rho, \]
we have
 \ben \sup _{0<t<T} \|v_1(\cdot, t)\|^2_{L^2(0,L)} & \le &  C \int_0^{\infty} \rho^{-2} |\h{k_1}^+(\rho) |^2
  \rho^2 d\rho\\
 &\le& C \int_0^{\infty} \mu^{-2/3} |\h{k_1}(i\mu)|^2 d\mu\\
  &\le&  C \|k_1\|_{H^{-\frac13}(\RR^+)}^2\\ &\leq & C\|\vec{k}\|_{{\cal K}_T}.
  \een
  Furthermore, for $\ell=-1, 0,1$, set $\mu = \rho ^3, \ \theta
  (\mu)=\mu ^{\frac13}$,

 \begin{eqnarray*}
 \partial_x^{\ell+1} v_1(x,t)&=&\frac{3}{\pi} \sum ^3_{j=1}\displaystyle
 {\cal R}e\  \int^{\infty}_0  (\lambda _j^+(\rho))^{\ell+1}e^{i\rho^3t}e^{ \l^+_j(\rho)x}
    \frac{\Delta_{j,1}^+(\rho)}{\Delta ^+(\rho)} \h{k}_1^+(\rho)
     \rho^2 d\rho\\
     &=& \frac{1}{\pi} \sum ^3_{j=1}\displaystyle
 {\cal R}e\  \int^{\infty}_0  (\lambda _j^+(\theta(\mu)))^{\ell+1}e^{i\mu t}e^{ \l^+_j(\theta(\mu))x}
    \frac{\Delta_{j,1}^+(\theta(\mu))}{\Delta ^+(\theta(\mu))} \h{k}_1(i\mu)
     d\mu
\end{eqnarray*}
Applying Plancherel Theorem in time $t$ yields that, for any $x\in
(0,L)$,
\begin{eqnarray*}
\| \partial_x^{\ell+1} v_1(x, \cdot)\|_{H^{-\frac{\ell}{3}} (0,T)} ^2&
\leq& C \sum ^3_{j=1} \int^{\infty}_0  \mu ^{-\frac{2\ell}{3}} \left
|(\lambda _j^+(\theta(\mu)))^{\ell+1}e^{ \l^+_j(\theta(\mu))x}
    \frac{\Delta_{j,1}^+(\theta(\mu))}{\Delta ^+(\theta(\mu))}
    \h{k}_1(i\mu) \right |^2 d\mu\\ &\leq & C
    \int^{\infty}_0 \mu ^{-\frac{2\ell}{3}} |\h{k}_1(i\mu )|^2
    d\mu \\ &\leq & C\|k_1\|_{H^{-\frac{\ell}{3}} (0,T)}^2\\ &\leq & C\|\vec{k}\|_{{\cal K}_T}^2.
    \end{eqnarray*}
for $\ell=-1,0,1$. Consequently
\[ \sup _{0<x<L}\| \partial_x^{\ell+1} v_1(x, \cdot)\|_{H^{-\frac{\ell}{3}}
(0,T)} \leq C\|\vec{k}\|_{{\cal K}_T} \] for $\ell=-1,0,1$. In
particular,
\[ \| v_1\|_{L^2 (0,T;  H^1(0,L))} \leq C\|\vec{k}\|_{{\cal K}_T}, \]
which ends the proof of Lemma \ref{e-1} for $v_1$. The proofs for
$v_2$ and $v_3$ are similar.\quad $\blacksquare$

\medskip
Now we turn to complete the proof of Proposition \ref{propauxakdv}.
It remains to prove that,
\begin{equation}\label{e2} \|\partial ^j_x u\|_{L^{\infty}_x (0,L;
H^{\frac{1-j}{3}}(0,T))} + \|\partial ^j_x z\|_{L^{\infty}_x (0,L;
H^{\frac{1-j}{3}}(0,T))}\leq C(\|w_0\|_{L^2 (0,T)} +\|f\|_{L^1
(0,T; L^2 (0,L))})
\end{equation}
for $j=0,1,2$.
\smallskip To this end, note that $u$ and $z$ can be written as
\[ u(t)= W_R (t) \tilde{w}_0-W_{bdr} (t) \vec{p}, \qquad z(t)= \int ^t_0
W_R (t-\tau )\tilde{f} (\tau ) d\tau - W_{bdr} (t) \vec{q}, \]
respectively. Here
\begin{itemize}
\item[(i)] $\tilde{w}_0$ and $\tilde{f}$ are zero extensions of
$w_0$ and $f$ from $(0,L)$ to $\R$:
\[ \tilde{w}_0(x) =\left \{ \begin{array}{ll} w_0 (x) & \ x\in
(0,L),\\  0& \ x\notin (0,L) \end{array} \right.  \quad
\tilde{f}_0(x,t) =\left \{ \begin{array}{ll} f(x,t) & \ (x, t)\in (0,L)\times (0,T),\\
0& \ x\notin (0,L). \end{array}  \right.\] \item[(ii)] $W_R (t) $
is the $C_0$ semigroup associated to the initial value problem
\[ \mu _t +\mu_{xxx}=0, \quad \mu(x,0)=\tilde{w}_0 (x), \qquad x\in
R, \ t\in (0,T) .\]
\item[(iii)]$\vec{p}= (p_1, p_2 , p_3)$ with
\[ p_1(t) =\mu _{xx}(0,t), \quad p_2 (t)=
\mu (L,t), \quad p_3 (t)= \mu _x (L,t) \] where $\mu (t)= W_R (t)
\tilde{w}_0 $.
 \item[(iv)] $\vec{q}=(q_1,
q_2, q_3)$ with
\[ q_1(t) =  \tilde{z}_{xx} (0,t) , \quad q_2 (t)=
\tilde{z}(L,t), \quad q_3 (t)=   \tilde{z}_x (L,t) \] where
$$\tilde{z}= \int ^t_0 W_R (t-\tau) \tilde{f} (\tau) d\tau .
$$
\end{itemize}
According to \cite{10RiKZ}, for $j=0,1,2$,
 \[ \|\partial ^j_x
\mu \|_{L^{\infty}_x (\RR; H^{\frac{1-j}{3}}(0,T))} \leq C
\|\tilde{w}_0\| _{L^2 (\RR)} \leq C \|w_0\|_{L^2 (0,L)}\] and
\[ \|\partial ^j_x
\tilde{z}\|_{L^{\infty}_x (\RR; H^{\frac{1-j}{3}}(0,T))} \leq C
\|\tilde{f}_0\| _{L^1 (0,T; L^2 (\RR))} \leq C \|f_0\|_{L^1 (0,T;
L^2 (0,L))}.\] Furthermore, by Lemma \ref{e-1},  \[ \|\partial
^j_x W_{bdr} (t) \vec{p}\|_{L^{\infty}_x (0,L;
H^{\frac{1-j}{3}}(0,T))} \leq C \|\vec{p}\|_{{\cal K}_T}\leq
C\|w_0\|_{L^2 (0,L)} \] and

\[ \|\partial ^j_x W_{bdr} (t)
\vec{q}\|_{L^{\infty}_x (0,L; H^{\frac{1-j}{3}}(0,T))} \leq C
\|\vec{q}\|_{{\cal K}_T}\leq C\|f\|_{L^1 (0,T; L^2 (0,L))}.  \]
The proof of Proposition \ref{propauxakdv} is thus complete. $\blacksquare$

\bigskip
\noindent {\bf Proof  of Lemma \ref{iff-1}}. As in the above proof
(see also \cite{03bona-sun-zhang}), the solution $\mu$ of \[ \left \{
\begin{array}{l} \mu_t +\mu _{xxx}=0, \quad \mu (x,0)=0 , \
(x,t)\in (0,L)\times (0,T), \\ \mu (0,t)=0, \quad \mu (L,t)
=g_2(t), \quad  \mu  _x (L,t)=0,\quad t\in(0,T)  \end{array} \right. \] can be
written as \[ \mu (x,t)=\mu _1 (x,t)+ \mu _2 (x,t) + \mu _3 (x,t) \]
with
\[ \mu _j (x,t) =\frac{3}{\pi} {\cal R}e \int ^{\infty}_0 e^{i\rho
^3t} e^{\lambda _j(\rho)x} S_j(\rho) \rho ^2 \hat{g}_2 (i\rho ^3)
d\rho\] for $j=1,2,3$ where
\[ \lambda _1 (\rho)=i\rho, \quad \lambda _2
(\rho)=\frac{\sqrt{3}}{2}\rho -\frac12 i\rho, \quad \lambda _3
(\rho)=-\frac{\sqrt{3}}{2}\rho -\frac12 i\rho,\]
\[ S_1 (\rho)\sim 1, \quad S_2 (\rho)\sim
e^{-\frac{\sqrt{3}}{2}L\rho}, \quad S_3 (\rho) \sim 1, \quad as \
\rho \to +\infty .\]  Arguing as before, $g_2\in H^{\frac13}
(0,T)$ implies that $\mu \in X_T$. On the other hand, if $\mu \in
X_T$, we show that we must have $g_2\in H^{\frac13} (0,T)$. First
note that as
\begin{eqnarray*} \mu _1 (x,t) &=& \frac{3}{\pi} {\cal
R}e\int ^{\infty}_0 e^{i\rho ^3 t}e^{i\rho x} S_1 (\rho ) \rho ^2
\hat{g}_2(i\rho ^3) d\rho\\
&=& \frac{1}{\pi} {\cal R}e \int ^{\infty}_0 e^{i\nu t} e^{i\nu
^{\frac13}x} S_1(\nu ^{\frac13}) \hat{g}_2(i\nu) d\nu
\end{eqnarray*}
and
\[ \partial _x \mu _1 (x,t)= \frac{1}{\pi} {\cal R}e \int ^{\infty}_0 e^{i\nu t} e^{i\nu
^{\frac13}x} \nu ^{\frac13}S_1(\nu ^{\frac13}) \hat{g}_2(i\nu)
d\nu ,\] it follows from the Plancherel  Theorem that for a constant $c>0$\[
\|\partial _x \mu _1\|_{L^2_x (0,L; L^2_t (R))}^2 =c
\|g_2\|_{H^{\frac13}(0,T)}^2 .\] Therefore, $\mu _1 \in L^2
(0,T; H^1 (0,L))$ if  and only if $g_2\in H^{\frac13} (0,T)$. Regarding
 $\mu _2 $, as
\[ \partial _x \mu _2 (x,t)= \frac{1}{\pi} {\cal R}e \int ^{\infty}_0 e^{i\nu t}
\exp \left (-\frac{\sqrt{3}}{2} \nu^{\frac13}(L-x)-\frac12 ix\nu
^{\frac13}\right ) \nu ^{\frac13}S_2(\nu
^{\frac13})e^{\frac{\sqrt{3}}{2} \nu^{\frac13}L} \hat{g}_2(i\nu)
d\nu ,\] we obtain
\[ \| \partial _x \mu _2 (x, \cdot )\|^2_{L^2_t (\R)} = c\int
^{\infty}_0 e^{-\sqrt{3}\nu ^{\frac13}(L-x)} \nu ^{\frac23}
|\hat{g}_2 (i\nu )|^2 d\nu \] and
\[ \int ^L_0 \| \partial _x \mu _2 (x, \cdot )\|^2_{L^2_t (\R)} dx
= c \int ^{\infty}_0 \nu ^{\frac13} |\hat{g}_2 (i\nu )|^2 d\nu =
c\|g_2\|_{H^{\frac16} (0,T)} .\] Similarly, we also have

\[ \int ^L_0 \| \partial _x \mu _3 (x, \cdot )\|^2_{L^2_t (\R)} dx
= c \int ^{\infty}_0 \nu ^{\frac13} |\hat{g}_2 (i\nu )|^2 d\nu =
c\|g_2\|_{H^{\frac16} (0,T)} .\]  Hence, $\mu_2 +\mu _3 \in X_T$ if
and only if $g_2\in H^{\frac16} (0,T)$. Consequently, $\mu \in X_T$
if and only if $g_2\in H^{\frac13} (0,T)$. The proof of Lemma
\ref{iff-1} is complete. $\blacksquare$



\subsection{Proof of Lemma \ref{next}}

If $N_T\neq \{ 0 \}$, then  the map $\va_T \in \CC N_T \to A(\va_T) \in \CC N_T$
 has at least one eigenvalue. Therefore there exist
$\lambda \in \CC$  and $\va_0 \in H^3(0,L) \setminus \{ 0\}$ such that
\be \label{vrpropio}\begin{cases}
\lambda \va_0 = - \va_0' - \va_0''',\\
\va_0(0)=0, \ \va_0'(0)=0, \  \va_0(L)+\va_0''(L)=0, \ \va'(L)=0.
\end{cases}
\ee

The solution of (\ref{vrpropio})  satisfies $\va_0(x)= \sum_{j=1}^3 C_j e^{\mu_j x}$
with  $\mu_j$ the roots of the polynomial  $$P(\mu) = \lambda + \mu + \mu^3. $$
More explicitly, they satisfy
\be\label{conditions}
\begin{cases}
 \mu_1+\mu_2+\mu_3=0\\
 \mu_1 \mu_2 +\mu_2 \mu_3 + \mu_3 \mu_1=1\\
 \mu_1 \mu_2 \mu_3=\lambda
\end{cases}
\ee

\noindent and $C_j$ for $j=1,2,3$  are the  solutions of the system
$$\left(\ba{ccc}
 1 & 1 & 1\\
\mu_1 & \mu_2 & \mu_3 \\
(1 + \mu_1^2) e^{\mu_1 L} & (1 + \mu_2^2) e^{\mu_2 L} & (1 + \mu_3^2) e^{\mu_3 L}\\
\mu_1 e^{\mu_1 L} & \mu_2 e^{\mu_2 L} & \mu_3 e^{\mu_3 L}
\ea  \right)
\left( \ba{c}  C_1 \\ C_2\\ C_3  \ea \right) =
\left( \ba{c}   0 \\ 0\\ 0        \ea \right).$$

Let us denote,  $a=L\mu_1$ and $ b=L\mu_2$.     Then by
\eqref{conditions},   $ c=L\mu_3 = -L(a+b)$ and  $$L^2=
-(a^2+ab+b^2).$$ Reducing the  rows of the matrix, one obtains the
new one

$$M:=\left(\ba{ccc}
 1 & 1 & 1\\
0 & b-a & -2a -b \\
0 & (L^2 +b^2) e^b - (L^2+a^2) e^a &  (L^2 +(a+b)^2)e^{-(a+b)} - (L^2 +a^2)e^a\\
0 & be^b - ae^a& -(a+b) e^{-(a+b)}-ae^a ,
\ea  \right) .$$

The  system has non-zero solutions if $\det(M) \neq 0$, which implies
\be\label{wq1}
\frac{(a+b) e^{-(a+b)}+ae^a}{2a+b} &=& \frac{b e^b - a e^a}{b-a},\\
\label{wq2}\frac{ab e^{-(a+b)}+  b(a+b)e^a}{2a+b} & =&\frac{ b(a+b)e^a-a(a+b)e^b}{b-a}.
\ee
Simplifying \eqref{wq1} one gets that
\ben
e^{-(a+b)}&=&\frac{2a+b}{b^2-a^2}b e^b - \frac{2b+a}{b^2-a^2}a e^a\\
            &=& \frac{(a+b)}{(b-a)}\left( \frac{2a+b}{b}e^b - \frac{2b+a}{a}e^a\right)
\een
and from \eqref{wq2}, we obtain

\ben
   (2a+b)(\frac{b}{b+a} -  \frac{a+b}{b}) e^b& =&  ( 2b+a)(\frac{a}{b+a}  - \frac{a+b}{a} ) e^a\\
 e^b&=& \frac{b^2}{a^2} e^a.
\een
Therefore, the set of non-zero solution is empty if and only if $L$ does not belong to

 \ben
\FF = \left\{ L\in \NN:  L^2=-(a^2+ab+b^2)  \ \text{with} \ a,b \in \CC^2 \ \text{satisfying}   \quad
\frac{e^a}{a^2}= \frac{e^b}{b^2} =\frac{e^{-(a+b)}}{(a+b)^2}\right\},
\een which concludes the proof of Lemma \ref{next}.


\begin{thebibliography}{100}

\bibitem{03bona-sun-zhang}
J.~L. Bona, \and S.~M. Sun, \and B.-Y. Zhang, \emph{A nonhomogeneous
  boundary-value problem for the {K}orteweg-de {V}ries equation posed on a
  finite domain}, Comm. Partial Differential Equations \textbf{28} (2003),
 1391--1436.

\bibitem{bsz-finite}  J.~L. Bona, \and S.~M. Sun, \and B.-Y. Zhang,
 \emph{Nonhomogeneous problem for the Korteweg-de Vries
equation in a bounded domain II}, J. Differential Equations, {\bf 247}
\,(2009), 2558--2596.

\bibitem{07cerpa}
E. Cerpa, \emph{Exact controllability of a nonlinear
{K}orteweg-de {V}ries
  equation on a critical spatial domain}, SIAM J. Control Optim. \textbf{46}
  (2007),  877--899 (electronic).

\bibitem{09cerpa-crepeau}
E. Cerpa, \and E. Cr{\'e}peau, \emph{Boundary
controllability for
  the nonlinear {K}orteweg-de {V}ries equation on any critical domain}, Ann.
  Inst. H. Poincar\'e Anal. Non Lin\'eaire \textbf{26} (2009), no.~2, 457--475.


\bibitem{c-g-1} T. Colin, \and
J.-M. Ghidaglia,\emph{ Un probl\`eme aux limites pour
l'\'equation de Korteweg-de Vries sur un intervalle born\'e.}
(French) [A boundary value problem for the Korteweg-de Vries
equation on a bounded interval] Journes ``Equations aux Drives
Partielles'' (Saint-Jean-de-Monts, 1997), Exp. No. III, 10 pp.,
\'Ecole Polytech., Palaiseau, 1997.

\bibitem{c-g-2} T. Colin, \and
J.-M. Ghidaglia, \emph{Un
probl\`eme mixte pour l'\'equation de Korteweg-de Vries sur un
intervalle born\'e. }(French) [A mixed initial-boundary value
problem for the Korteweg-de Vries equation on a bounded interval]
C. R. Acad. Sci. Paris S\'er. I Math. \textbf{324} (1997),
599--603.

\bibitem{ColGhi01} T. Colin, \and
J.-M. Ghidaglia, \emph{An initial-boundary-value
problem fo the Korteweg-de Vries Equation posed on a finite
interval}, Adv. Differential Equations \textbf{6} (2001),
1463--1492.





\bibitem{04coron-crepeau}
J.-M. Coron,\and  E. Cr{\'e}peau, \emph{Exact boundary
  controllability of a nonlinear {K}d{V} equation with critical lengths}, J.
  Eur. Math. Soc. (JEMS) \textbf{6} (2004), no.~3, 367--398.



\bibitem{08glass-guerrero}
O. Glass, \and S. Guerrero, \emph{Some exact
controllability results for
  the linear {K}d{V} equation and uniform controllability in the
  zero-dispersion limit}, Asymptot. Anal. \textbf{60} (2008), no.~1-2, 61--100.

\bibitem{09glass-guerrero} O. Glass, \and S. Guerrero, \emph{Controllability of
the Korteweg-de Vries equation from the right Dirichlet boundary condition}, Systems Control Lett. \textbf{59} (2010), no. 7, 390--395.


\bibitem{guilleron} J.-P. Guilleron, \emph{Null controlability of a linear KdV equation on an interval with
special boundary conditions}, preprint.


\bibitem{Kato83} T. Kato, \emph{On the Cauchy problem for the
(generalized) Korteweg-de Vries equations}, Advances in
Mathematics Supplementary Studies, Studies in Applied Math.
\textbf{8} (1983), 93--128.



\bibitem{kz} E. Kramer, \and  B.-Y. Zhang,  \emph{ Nonhomogeneous boundary
value problems for the Korteweg-de Vries equation on a bounded
domain}, J. Syst. Sci. Complex,  \textbf{23}\,(2010), 499--526.

\bibitem{10RiKZ}
E. Kramer, \and I. Rivas, \and B.-Y. Zhang,
\emph{Well-posedness of a class
  of initial-boundary-value problem for the {K}ortweg-de {V}ries equation on a
  bounded domain}, ESAIM Control Optim. Calc. Var., to appear.


\bibitem{lrz} C. Laurent, L. Rosier, B.-Y.Zhang,  \emph{Control and Stabilization of the Korteweg-de Vries
Equation on a Periodic Domain}, Comm. Partial Diff. Eqns.
\textbf{35}(2010), 707--744.




\bibitem{Pazy} A. Pazy, \emph{Semigroups of linear operators and applications to
partial differential equations}, Applied Mathematical Sciences,
{\bf 44}, Springer-Verlag, New York, 1983.

\bibitem{perla} G. Perla-Menzala, C. F. Vasconcellos, \and E. Zuazua, \emph{Stabilization of the Korteweg-de Vries equa-
tion with localized damping,}  Quart. Appl. Math., 2002, {\bf 60}
(2002),  111--129.

\bibitem{pazoto} A. F. Pazoto, \emph{Unique continuation and decay for the Korteweg-de Vries equation with localized
damping,}  ESAIM Control Optim. Calc. Var., {\bf 11} (2005),
473--486.


\bibitem{ruz} I. Rivas, \and M. Usman, \and  B.-Y. Zhang, \emph{Global well-posedness and asymptotic behavior of
a class of initial-boundary-value problem of the
 Korteweg-de Vries equation on a finite domain},  Mathematical
 Control and Related Fields {\bf 1} (2011), 1, 61--81.


 \bibitem{97rosier} L. Rosier, \emph{Exact boundary controllability for the
{K}orteweg-de {V}ries equation on a bounded domain}, ESAIM Control
Optim. Calc. Var.  \textbf{2} (1997), 33--55 (electronic).

\bibitem{04rosier}
  L. Rosier, \emph{Control of the surface of a fluid by a wavemaker},
ESAIM Control
  Optim. Calc. Var. \textbf{10} (2004), no.~3, 346--380 (electronic).




\bibitem{rosier-z}  L. Rosier, \and B.-Y. Zhang,  \emph{Global stabilization
of the generalized Korteweg-de Vries equation,} SIAM J. Control
Optim., \textbf{45} (2006), 927--956.

\bibitem{Za-survey} L. Rosier, \and  B.-Y. Zhang, {\em Control and
stabilization of the Korteweg-de Vries equation: recent
progresses}, J. Syst. Sci. Complex,  {\bf 22}\,(2009), 647--682.



\bibitem{simon}
J. Simon, \emph{Compact sets in the space {$L^p(0,T;B)$}}, Ann. Mat.
Pura Appl. (4) \textbf{146} (1987), 65--96.




\bibitem{zh} B.-Y. Zhang,  \emph{Boundary stabilization of the Korteweg-de Vries  equations}, Proc. of
International Conference on Control and Estimation of Distributed
Parameter Systems: Nonlinear Phenomena, Vorau (Styria, Austria),
July 18-24, 1993,  International Series of Numerical Mathematics,
\textbf{118}, 371 -- 389, 1994.




\bibitem{Za} B.-Y. Zhang,  \emph{Exact boundary controllability of the
Korteweg-de Vries equation}, SIAM J. Cont. Optim., {\bf 37}
(1999), 543--565.

\bibitem{zhang-2} B.-Y. Zhang, \emph{Well-posedness and control of the Korteweg-de Vries equation on a bounded domain,}
Fifth International Congress of Chinese Mathematicians, Part 1, 2,
931--956, AMS/IP Stud. Adv. Math., \textbf{51}, pt. 1, 2, Amer.
Math. Soc., Providence, RI, 2012.

\bibitem{zhang-1} B.-Y. Zhang,   \emph{Hidden regularities of solutions of the
Korteweg-de Vries equation and their applications}, Preprint.
\end{thebibliography}

\end{document}